\documentclass{amsart}
\usepackage{enumerate,amssymb,amsmath,graphics,epsfig}
\usepackage[colorlinks]{hyperref}
\usepackage{color}
\usepackage{xspace}
\usepackage[labelfont=rm,font=small,labelformat=simple]{subcaption}

\newtheorem{theorem}{Theorem}[section]

\newtheorem{lemma}[theorem]{Lemma}
\newtheorem{proposition}{Proposition}

\newtheorem{problem}{Problem}
\theoremstyle{definition}

\newtheorem{remark}{Remark}

\newcommand\RE{\mathbb{R}}

\renewcommand\div{\mathop{\rm{div}}\nolimits}
\newcommand\Grad{\mathop{\boldsymbol\nabla}\nolimits}
\newcommand\Grads{\mathop{\boldsymbol\nabla_{\rm sym}}\nolimits}
\newcommand\grad{\mathop\nabla\nolimits}
\newcommand\Huo{H^1_0(\Omega)}
\newcommand\Ldo{L^2_0(\Omega)}

\newcommand\B{\mathcal B}
\renewcommand\u{\mathbf{u}}
\renewcommand\v{\mathbf{v}}
\renewcommand\d{\mathbf{d}}
\renewcommand\c{\mathbf{c}}

\newcommand\x{\mathbf{x}}
\newcommand\s{\mathbf{s}}
\newcommand\X{\mathbf{X}}
\newcommand\Y{\mathbf{Y}}

\newcommand\FF{\mathbf{F}}
\newcommand\ssigma{\boldsymbol\sigma}
\newcommand\pphi{\boldsymbol\varphi}
\newcommand\zzeta{\boldsymbol\zeta}
\newcommand\cchi{\boldsymbol\chi}
\newcommand\llambda{\boldsymbol\lambda}
\newcommand\LLambda{\boldsymbol\Lambda}
\newcommand\mmu{\boldsymbol\mu}
\newcommand\dirac{\boldsymbol\delta}
\newcommand\F{\mathbb{F}}
\renewcommand\P{\mathbb{P}}
\newcommand\OT{\left]0,T\right[}
\newcommand\ds{\mathrm{d}\s}

\newcommand\dA{\mathrm{dA}}
\newcommand{\dr}{\delta\rho}
\newcommand\Hub{(H^1(\B))^d}

\newcommand\dt{\Delta t}
\newcommand\Vh{\mathbf{V}_h}
\newcommand\Qh{Q_h}
\newcommand\Sh{\mathbf{S}_h}
\newcommand\Lh{\LLambda_h}
\newcommand\lh{\llambda_h}

\newcommand\f{\mathbf{f}}
\newcommand\g{\mathbf{g}}

\newcommand\V{\mathbf{V}}

\newcommand\tria{\mathcal{T}_h}
\newcommand\triaB{\mathcal{T}_h^{\B}}
\newcommand\ucX{\u(\X(\cdot,t),t)}
\newcommand\uhcX{\u_h(\X_h(\cdot,t),t)}
\newcommand\uhcXn{\u_h^{n+1}(\X_h^n)}
\newcommand\vcX{\v(\X(t))}
\newcommand\vcXn{\v(\X_h^n)}
\newcommand{\jump}[1]  {[\![ #1 ]\!]}
\newcommand\PP{\mathbb{P}}
\newcommand\PPt{\tilde\PP}
\newcommand\Acca{H}

\newcommand\ibm{\textsc{ibm}\xspace}
\newcommand\feibm{\textsc{fe-ibm}\xspace}
\newcommand\dlmibm{\textsc{dlm-ibm}\xspace}
\newcommand\cfl{\textsc{cfl}\xspace}

\newcommand\N{\mathbf{N}}
\newcommand{\Eh}{\mathcal{E}_h}

\begin{document}
\title[Finite element immersed boundary method] %Use the shortened version of the full title
      {Discrete models for fluid-structure interactions: the Finite Element
       Immersed Boundary Method}

\dedicatory{Dedicated to Paolo Secchi and Alberto Valli on the occasion of
their sixtieth birthdays}

\author{Daniele Boffi}
\address{Dipartimento di Matematica ``F. Casorati'', Universit\`a di Pavia,
Italy}
\email{daniele.boffi@unipv.it}
\urladdr{http://www-dimat.unipv.it/boffi/}

\author{Lucia Gastaldi}
\address{DICATAM, Universit\`a di Brescia, Italy}
\email{lucia.gastaldi@unibs.it}
\urladdr{http://www.ing.unibs.it/gastaldi/}


%please go to the last acknowlegments section.
\thanks{The authors are supported by PRIN-MIUR grant and by
GNCS (Istituto Nazionale di Alta Matematica) grant}

%The abstract of your paper
\begin{abstract}
The aim of this paper is to provide a survey of the state of the art in the
finite element approach to the Immersed Boundary Method (\feibm) which has
been investigated by the authors during the last decade. In a unified setting,
we present the different formulation proposed in our research and highlight
the advantages of the one based on a distributed Lagrange multiplier
(\dlmibm) over the original \feibm.
\end{abstract}

\maketitle

%The title of your section 1
\section{Introduction}
\label{se:intro}
In this paper we present in a unified setting some results which have been the
object of our research on the finite element discretization of the Immersed
Boundary Method (\ibm) during the last years. This is a survey paper which
aims at providing a comprehensive and self-contained presentation of this
subject.

The \ibm has been introduced by Peskin in the seventies~\cite{Pe1,Pe2} for
the numerical approximation of biological phenomena involving fluids and
solids (typically, the blood flow in the cardiac muscle). Since then, it has
been successfully adopted in many application areas; the interested reader
can refer
to~\cite{PeAN,Leveque1988191,Fauci198885,Zhu2002452,Givelberg2003377,
Givelberg20040283,Miller2005195,Kim20062294,Heys2008977,Kim2009927,
Griffith2012433} and the references therein.
One of the main features of the \ibm is that the solid is thought as a part of
the fluid and its effect on the dynamics of the system is modeled through a
Dirac delta functions that has the role of linking Eulerian and Lagrangian
variables. The original discretization of the \ibm, which makes use of finite
differences for the underlying fluid equations, needs a suitable approximation
of the delta function. This procedure is a key point of the numerical
strategy: tuning the approximation of the delta function is a crucial aspect
which has a great influence on how the discontinuities of the solution are
well captured.

In this framework it is natural to consider a finite element version of the
\ibm, where the delta function does not need any approximation, since it can
be handled directly by the variational formulation. The first steps in this
direction have been presented in~\cite{bathe}. We refer to this formulation as
\feibm~\cite{bgh1,bgh2,heltai,BGHP,BCG2011}. We recall that our formulation
can model systems where the solid and the fluid have the same dimension
(i.e., codimension zero case), or where the solid has codimension one.
There are also other finite element approaches to the \ibm: some of them
handle the delta function variationally as we
do~\cite{heltaicostanzo,hoppe1,hoppe2}, other ones use different techniques
based on its approximation~\cite{WangLiu,ZGWangLiu,LiuKimTang}.

During our studies concerning this topic, we have discussed two important
issues: the mass conservation of the overall procedure and the stability of
the fully discrete scheme.
Concerning the mass conservation, we have observed that, as expected,
discontinuous pressure schemes can enforce the mass conservation (expressed by
the fact that the fluid velocity is divergence free) much more locally than
continuous ones. In this respect, we have analyzed a modification of
continuous pressure schemes which improves their
performances~\cite{bcgg2012,bcggumi}. Several numerical tests confirm our
theoretical investigations~\cite{coupled2011}.
In our scheme, we use a semi-implicit time advancing scheme.
For its stability, we have observed that the semi-implicit
\feibm is subjected to a \cfl condition which depends on the parameters
involved with the model; in particular, it has been shown that the method is
\cfl-stable even when the densities of the fluid and the solid are comparable.
In this respect, the \feibm is superior to the more popular Arbitrary
Lagrangian Eulerian (\textsc{ale}) method which has been proved to be
unconditionally unstable unless a fully implicit approach is
considered~\cite{causin}.

Recently, we have studied a modification of the \feibm consisting in the
introduction of a Lagrange multiplier associated to the equation coupling
solid and fluid velocities. Since the resulting formulation resembles the
fictitious domain method with distributed Lagrange multiplier, we refer to the
new scheme as \dlmibm (see~\cite{glokuz2007,girglo1995,girglopan}).
First promising numerical results have been presented
in~\cite{coupled2013}, showing that the (semi-implicit) scheme is
unconditionally stable and (surprisingly) enjoys better conservation
properties than the previous one.
In~\cite{nuovodlm} we performed a stability analysis for the \dlmibm, showing
rigorously that the semi-implicit scheme is unconditionally stable, and
reporting on some additional numerical tests.

The structure of the paper is as follows: in Section~\ref{se:ibm} we recall
the fundamental properties of the original \ibm; in Section~\ref{se:fe-ibm} we
describe the \feibm and the variational treatment of the Dirac delta function;
in Section~\ref{se:DLM} we show how to introduce the Lagrange multiplier which
has the effect of stabilizing the time advancing scheme; in
Section~\ref{se:stab} we discuss the stability of the proposed scheme and
in Section~\ref{se:num} we report some numerical experiments.

\section{The Immersed Boundary Method}
\label{se:ibm}
In this section we review the formulation of the \ibm in
its original version introduced by Peskin in~\cite{PeAN} for immersed
materials modeled as collections of fibers  and further extended 
to cover the case of \emph{thick} bodies modeled as a hyperelastic material
in~\cite{BGHP}. One of the main difficulties to face when
treating fluid-structure interaction problems consists in the fact that the
fluid is naturally modeled using Eulerian variables, while for the solid the
Lagrangian framework is more appropriate. The \ibm is a way to
mix Eulerian and Lagrangian variables thanks to the use of the Dirac delta
function. The main idea in \ibm is to consider the structure as a part of the
fluid where additional mass and forces are concentrated.

Let $\Omega\subseteq\RE^d$, $d=2,3$ be a region containing both a viscous
incompressible fluid and an immersed elastic structure. We introduce
the Navier--Stokes equations describing the dynamics of the fluid with respect
to the Eulerian variable denoted by $\x$:
\begin{equation}
\label{eq:NS}
\aligned
&\rho\left(\frac{\partial\u}{\partial t}+\u\cdot\Grad\u\right)
-\nu\Delta\u+\Grad p=\FF&&\text{ in }\Omega\times]0,T[\\
&\div\u=0&&\text{ in }\Omega\times]0,T[.
\endaligned
\end{equation}
Here $\rho$ and $\nu$ are positive constants denoting the density and the
viscosity of the fluid. The unknowns $\u$ and $p$ represent the velocity and the
pressure of the fluid, respectively. The right hand side $\FF$ stands for the
forces acting on the fluid, and, in absence of external volume forces, it takes
into account the so called \emph{fluid-structure interaction forces}, that is
the forces exerted by the elastic structure on the fluid.

The immersed structure is considered as an elastic incompressible material
filling at time $t$ a region $\B_t\subseteq\Omega$ of codimension one or zero.
Using Lagrangian variables, $\B_t$ can be represented as the image of a mapping 
$\X$ from a reference domain $\B\subset\RE^m$, with $m=d$ or $m=d-1$.
We denote by $\s$ the Lagrangian coordinates
in the reference domain $\B$, then $\X(\s,t)$ represents the position of a
point in the current solid domain $\B_t$ which is labeled $\s$ in the
reference domain, that is:
\[
\aligned
&\X(\cdot,t):\B\to\B_t\\
&\x\in\B_t \iff \x=\X(\s,t)\quad\text{for some }\s\in\B.
\endaligned
\]
Since $\u(\x,t)$ represents the velocity of a material point at position 
$\x$ at time $t$, we have the following condition
\begin{equation}
\label{eq:velocity}
\u(\x,t)=\frac{\partial\X}{\partial t}(\s,t)\quad\text{for }\x=\X(\s,t).
\end{equation}
The source term in the first equation of~\eqref{eq:NS} can now be written 
in absence of external volume forces as follows:
\begin{equation}
\FF(\x,t)=\int_{\B}\f(\s,t)\delta(\x-\X(\s,t))\,d\s\quad
\text{in }\Omega\times]0,T[,
\label{eq:sourcedelta}
\end{equation}
where $\f(\s,t)$ is the force density that the immersed material exerts to the
fluid and that can be modeled in different ways depending on the application
field. In order to give an idea, we report here the most simple example of an
elastic structure represented by a massless closed curve $\Gamma_t$ immersed in
the fluid occupying a two dimensional domain $\Omega$. The curve is given in
parametric form as $\X(\s,t)$ for $0\le\s\le L$, with $\X(0,t)=\X(L,t)$. The
local density force applied by the curve to the fluid is given by
$\f=\partial(T\tau)\big/\partial\s$, where $T$ is the tension and $\tau$ is the
unit tangent to the curve.
Assuming that $T$ depends linearly on $|\partial\X/\partial\s|$ we obtain 
\begin{equation}
\f(\s,t)=\kappa\frac{\partial^2\X}{\partial\s^2}.
\label{eq:forzaelastica}
\end{equation}
To summarize, the \ibm formulation for
fluid-structure interaction problems has the following form.
\begin{problem}
\label{pb:peskin}
Find $\u:\Omega\times]0,T[\to\RE^d$, $p:\Omega\times]0,T[\to\RE$,
$\X:\B\times]0,T[\to\Omega$ which satisfy:
\begin{alignat}{2}
&\rho\left(\frac{\partial\u}{\partial t}+\u\cdot\grad\u\right)-\nu\Delta\u
+\grad p=\FF&&\quad
\mbox{in }\Omega\times]0,T[
\label{eq:NavierStokes}\\
&\div\u=0&&\quad
\mbox{in }\Omega\times]0,T[
\label{eq:NavierStokesdiv}\\
&\FF(\x,t)=\int_{\B}\f(\s,t)\delta(\x-\X(\s,t))d\s
&&\quad\mbox{in }\Omega\times ]0,T[
\label{eq:forza}\\
&\frac{\partial\X}{\partial t}=\u(\X(\s,t),t)
&&\quad\mbox{in }\Omega_0\times]0,T[
\label{eq:noslip}\\
&\u(\x,t)=0&& \quad\mbox{on }\partial\Omega\times]0,T[
\label{eq:bcu}\\
&\u(\x,0)=\u_0(\x)&&\quad\mbox{in }\Omega
\label{eq:initu}\\
&\X(\s,0)=\X_0(\s)&&\quad\mbox{in }\B.
\label{eq:initX}
\end{alignat}
\end{problem}
The above formulation has been derived in~\cite{PeAN} by using the
principle of least action and it is well suited to the case of a structure
described by a system of elastic fibers.
In order to treat more general elasticity models for thick structure, it
has been observed in~\cite{BGHP} that an additional transmission condition
along the interface between the immersed body and the surrounding fluid is
needed. We give here the formulation of \ibm obtained in~\cite{BCG2011} extending
the formulation of~\cite{BGHP} to the case of fluid and solid with different
densities. 

The strong form of the equation of motion can be written as follows
\begin{equation}
\label{eq:motion}
\rho\dot\u=\rho\left(\frac{\partial \u}{\partial t}+\u\cdot\Grad\u\right)
=\div\ssigma
\quad \text{in }\Omega,
\end{equation}
where $\ssigma$ stands for the Cauchy stress tensor. 

Let us consider first the case $m=d$ and let us introduce some assumption
describing the characteristics of fluid and solid materials.
In the fluid, $\ssigma$ is modeled by means of the Navier--Stokes stress tensor
$\ssigma_f=-p{\mathbb I}+\nu(\Grad\u+(\Grad\u)^\top)$. 
We assume that the solid is composed by a viscous hyperelastic
material, therefore $\ssigma$ can be decomposed as the sum of the viscous part
$\ssigma_f$ and an elastic part $\ssigma_s$, which takes into account the
elastic behavior of the material.
Hence, the Cauchy stress tensor can be written as follows
\begin{equation}
\label{eq:stress-decomposition}
\ssigma = \left\{
\begin{array}{ll}
\ssigma_f & \text{ in }\quad \Omega \backslash \B_t \\
\ssigma_f + \ssigma_s \quad& \text{ in } \quad \B_t.
\end{array}
\right.
\end{equation}
Since in the description of the deformation of an elastic body the Lagrangian
setting is more convenient, we express $\ssigma_s$ in Lagrangian variables:
this can be done by introducing the first Piola--Kirchhoff stress tensor
$\PPt$ defined in such a way that, for any arbitrary smooth portion
$\mathcal{P}$ of $\B$ evolving as $\mathcal{P}_t = \X(\mathcal{P},t)$, it
holds
\begin{equation}
\label{eq:pk-definition}
\int_{\partial\mathcal{P}_t}\ssigma_s\mathbf{n} \mathrm{da}=
\int_{\partial\mathcal{P}}\PPt\mathbf{N}\,\mathrm{dA}
\quad\text{for all }\mathcal{P}_t;
\end{equation}
where $\mathbf{N}$ is the outer normal to the region $\mathcal{P}$ in the
Lagrangian coordinates. The first Piola--Kirchhoff stress tensor gives the
\emph{elastic} force per unit reference volume ($d=3$) or area
($d=2$), expressed in the reference space, and its pointwise
expression is given by
\begin{equation}
\label{eq:pk-point-wise}
\PPt(\s,t) = |\FF(\s,t)| \, \ssigma_s(\X(\s,t),t) \, \FF^{-\top}(\s,t).
\end{equation}
Moreover, the densities of the fluid and the solid could be different so that 
we set
\begin{equation}
\label{eq:density}
\rho=\left\{
\begin{array}{ll}
\rho_f & \text{in }\Omega\setminus\B_t\\
\rho_s & \text{in }\B_t. 
\end{array}
\right.
\end{equation}

The case of the immersed body occupying a region of codimension one represents
a mathematical simplification of a thin body with thickness $t_s$ very small
with respect to the other space dimensions, so that one can assume that the
physical quantities depend only on the variables along the middle section of
the body represented by $\B_t$ and are constant in the orthogonal direction.
In this case the thickness $t_s$ appears as a multiplicative factor in the
expressions of the Piola--Kirchhoff stress tensor and 
of the density of the solid (see~\cite{BCG2011} for the details). 
In order to unify the formulation of the problem we set
\begin{equation}
\label{eq:unification}
\dr=\left\{\begin{array}{ll}
\rho_s-\rho_f &\text{if }\dim{\B_t}=d\\
(\rho_s-\rho_f)t_s&\text{if }\dim{\B_t}=d-1
\end{array}
\right.
\quad
\PP=\left\{\begin{array}{ll}
\PPt &\text{if }\dim{\B_t}=d\\
t_s\PPt&\text{if }\dim{\B_t}=d-1.
\end{array}
\right.
\end{equation}
Using the above definitions~\eqref{eq:pk-definition}, \eqref{eq:density}
and~\eqref{eq:unification}
in~\eqref{eq:motion}, the principle of virtual work provides with some
computations the following strong form of the problem.
\begin{problem}
Find $\u:\Omega\times\OT\to\RE^d$, $p:\Omega\times\OT\to\RE$ and
$\X:\B\times\OT\to\Omega$ which satisfy:
\begin{equation}
\aligned
 &\rho_f\left(\frac{\partial\u}{\partial t}+\u\cdot\grad\u\right)
   -\nu\Delta\u
   +\grad p =\d+\g+\mathbf{t} &&\quad \mbox{in }\Omega\times]0,T[
      \\
 &\div\u=0&&\quad \mbox{in }\Omega\times]0,T[
     \\
 &\d(\x,t)=
   -\dr\int_{\B}\frac{\partial^2\X}{\partial t^2}
   \dirac(\x-\X(\s,t))\ds
   &&\quad\mbox{in }\Omega\times ]0,T[
    \\
 &\g(\x,t)=\int_{\B}\Grad_s\cdot\PP\dirac(\x-\X(\s,t))\ds
   &&\quad\mbox{in }\Omega\times ]0,T[
   \\
 &\mathbf{t}(\x,t)=-\int_{\partial \B}\PP\mathbf{N}\dirac(\x-\X(\s,t))\dA
   &&\quad\mbox{in }\Omega\times ]0,T[
\\
 &\frac{\partial\X}{\partial t}(\s,t)=\u(\X(\s,t),t)
   &&\quad\mbox{in }\B\times]0,T[
  \\
 &\u(\x,t)=0&& \quad\mbox{on }\partial\Omega\times]0,T[
 \\
 &\u(\x,0)=\u_0(\x)&&\quad\mbox{in }\Omega
\\
 &\X(\s,0)=\X_0(\s)&&\quad\mbox{in }\B.
\endaligned
\end{equation}
\label{pb:cont-original}
\end{problem}

\section{The finite element Immersed Boundary Method}
\label{se:fe-ibm}

In this section we review the history of our approach to the finite element
Immersed Boundary Method \feibm.

The starting point of our analysis has been introduced in~\cite{bathe} (see
also~\cite{bgh1}).
The main idea is that the source term which represents the effects of the
structure on the fluid and which involves the presence of a Dirac delta
function, can be naturally written in variational form. This leads to a
variational formulation of the Immersed Boundary Method which can be used
efficiently for the finite element discretization. In order to describe this
formulation, let us consider the previously introduced source term defined in
terms of the Dirac delta function (see~\eqref{eq:sourcedelta}):
\begin{equation}
\FF(\x,t)=\int_{\B}\f(\s,t)\delta(\x-\X(\s,t))\,d\s\quad
\text{in }\Omega\times]0,T[,
\label{eq:forzadelta}
\end{equation}
where the function $\f$ is related to elastic properties of the solid
expressed in the Lagrangian variable $\s$. The following Lemma shows that
$\FF$ can actually be interpreted as an element of $H^{-1}(\Omega)$.

\begin{lemma}[see~\cite{bathe,bgh1}]

Let $\B_t$ be Lipschitz continuous for all $t\in[0,T]$ and suppose that $\f$
belongs to $L^2(\B)$ for all $t\in]0,T[)$.
Then $\FF$ (see~\eqref{eq:forzadelta}) is a
distribution belonging to $H^{-1}(\Omega)$ defined as
\[
{}_{H^{-1}(\Omega)}\langle\FF(t),\v\rangle_{H^1_0(\Omega)}=
\int_{\B}\f(\s,t)\cdot\v(\X(\s,t))\,d\s\quad\forall t\in]0,T[,\
\forall\v\in(H^1_0(\Omega))^d.
\]

\end{lemma}

This observation made it possible to introduce the following variational
formulation for the first \feibm approach. We start describing this model in
the case of a constant density $\rho$ and when the elastic force can be
modeled as
\[
\f(\s,t)=\kappa\frac{\partial^2\X}{\partial\s^2}(\s,t),
\]
$\kappa$ being the elastic constant of the solid
(see~\eqref{eq:forzaelastica}).

\begin{problem}

Given $\u_0\in(H^1_0(\Omega))^d$ and $\X_0:\B\to\Omega$, for almost every
$t\in]0,T[$ find $(\u(t),p(t))\in(H^1_0(\Omega))^d\times L^2_0(\Omega)$ and
$\X:\B\to\Omega$ such that
\begin{equation}
\aligned
&\rho\left(\frac{d}{dt}(\u(t),\v)+(\u(t)\cdot\Grad\u(t),\v)\right)&&\\
&\qquad+\nu(\Grads\u(t),\Grads\v)-(\div\v,p(t))=\langle(\FF(t),\v\rangle\quad&&
\forall\v\in(H^1_0(\Omega))^d\\
&(\div\u(t),q)=0&&\forall q\in L^2_0(\Omega)\\
&\langle\FF(t),\v\rangle=\int_{\B}\kappa\frac{\partial^2\X(\s,t)}{\partial s^2}
\cdot\v(\X(\s,t))\,d\s&&\forall\v\in(H^1_0(\Omega))^d\\
&\frac{\partial\X}{\partial t}(\s,t)=\u(\X(\s,t),t)&&\forall\s\in\B\\
&\u(\x,0)=\u_0(x)&&\forall\x\in\Omega\\
&\X(\s,0)=\X_0(\s)&&\forall\s\in\B.
\endaligned
\end{equation}
\label{pb:ibmvar1}
Here $\Grads=\Grad+(\Grad)^{\top}$ is the symmetric gradient.
\end{problem}

As described in Section~\ref{se:ibm}, the first model has been later modified
to include the case of a general hyperelastic material and different densities
in the fluid and the solid~\cite{BGHP,BCG2011}.
The variational formulation of Problem~\ref{pb:cont-original} reads
\begin{problem}
Given $\u_0\in(\Huo)^d$ and $\X_0:\B\to\Omega$, for almost every $t\in\OT$,
find $(\u(t),p(t))\in(\Huo)^d\times\Ldo$ and $\X(t) \in W^{1,\infty}(\B)$,
such that
\begin{equation}
\label{eq:ibmvar}
\aligned
  &\rho_f\left(\frac{d}{dt}(\u(t),\v)+
   \left((\u(t)\cdot\Grad\u(t),\v)-(\u(t)\cdot\Grad\v,\u(t))\right)/2\right)\\
  &\qquad+\nu(\Grads\u(t),\Grads\v)-(\div\v,p(t))=
   \langle\d(t),\v\rangle+\langle\FF(t),\v\rangle\quad&&
   \forall\v\in(\Huo)^d\\
  &(\div\u(t),q)=0&&\forall q\in\Ldo\\
  &\langle\d(t),\v\rangle=
   -\dr\int_\B\frac{\partial^2\X}{\partial t^2}\v(\X(\s,t))\,\ds
      &&\forall\v\in(\Huo)^d\\
  &\langle\FF(t),\v\rangle= - \int_\B \PP(\F(\s,t)): \Grad_s
\v(\X(\s,t))\,\ds
      &&\forall\v\in(\Huo)^d\\
  &\frac{\partial\X}{\partial t}(\s,t) =\u(\X(\s,t),t) &&\forall \s\in \B\\
  &\u(\x,0)=\u_0(\x)\ &&\forall \x\in\Omega\\
  &\X(\s,0)=\X_0(\s)\ &&\forall \s\in\B.
\endaligned
\end{equation}
\label{pb:ibmvar}
\end{problem}
In the last equation we have used the previously introduced definitions for
the scaled density difference
\[
\dr=\left\{\begin{array}{ll}
\rho_s-\rho_f &\text{if }\dim{\B_t}=d\\
(\rho_s-\rho_f)t_s&\text{if }\dim{\B_t}=d-1
\end{array}
\right.
\]
and the scaled Piola--Kirchhoff tensor
\[
\P=\left\{\begin{array}{ll}
\PPt &\text{if }\dim{\B_t}=d\\
t_s\PPt&\text{if }\dim{\B_t}=d-1,
\end{array}
\right.
\]
$\PPt$ being the standard Piola--Kirchhoff tensor (see~\eqref{eq:unification}).

We are now ready to describe a finite element formulation associated with
Problem~\ref{pb:ibmvar}. We start with the space semidiscretization which is
a more or less immediate consequence of our variational formulation.

Let $\tria$ be a standard triangulation of $\Omega$. It is important to remark
that this triangulation will never change during the computation. This is one
of the main differences with respect to other strategies, such as the
Arbitrary Lagrangian Eulerian (ALE) approach, where a moving mesh is
considered.

Then we consider two Stokes-stable finite element spaces $V_h\subset\Huo^d$
and $Q_h\subset\Ldo$.

We need a second mesh for the solid: this will be constructed on the reference
configuration $\B$ and will be denoted by $\triaB$. In general, we use
simplicial meshes for the solid and the corresponding finite element space is
defined as
\[
\Sh=\{\Y\in C^0(\B;\Omega):\Y|_{T_k}\in\mathcal{P}^1(T_k)^d,\ k=1,\dots,M_e\},
\]
where $T_k$ denotes an element of $\triaB$ and $M_e$ is the total number of
such elements.

It turns out that the space semidiscretization of Problem~\ref{pb:ibmvar}
reads
\begin{problem}
Given $\u_{h0}\in\Vh$ and $\X_{h0}\in\Sh$, for almost every $t\in\OT$, find
$(\u_h(t),p_h(t))\in\Vh\times\Qh$ and $\X_h(t)\in\Sh$, such that
\begin{alignat*}{2}
&\rho_f\left(\frac d {dt}(\u_h(t),\v)+((\u_h(t)\cdot\Grad\u_h(t),\v)-
(\u_h(t)\cdot\Grad\v,\u_h(t))/2\right)\\
&\qquad+\nu(\Grads\u_h(t),\Grads\v)\\
&\qquad-(\div \v,p_h(t))=\langle\d_h(t),\v\rangle+\langle\FF_h(t),\v\rangle
&&\forall \v\in\Vh\\
&(\div\u_h(t),q)=0&&\forall q\in\Qh\\
&\langle\d_h(t),\v\rangle=-\dr\int_\B\frac{\partial^2\X_h}{\partial
t^2}(t)\v(\X_h(t))\ds
&&\forall \v\in\Vh\\
&\langle\FF_h(t),\v\rangle=-\int_\B\P(\F_h(\s,t)):\Grad_s\v(\X(\s,t))\,d\s
&&\forall \v\in\Vh\\
&\u_h(\x,0)=\u_{h0}(\x)\ &&\forall \x\in\Omega\\
&\frac{\partial\X_{hi}}{\partial t}(t)=\u_h(\X_{hi}(t),t)
&&\forall i=1,\dots,M\\
&\X_{hi}(0)=\X_{h0}(\s_i) &&\forall i=1,\dots,M,
\end{alignat*}
where $M$ denotes the number of degrees of freedom in the space $\Sh$, $\s_i$
is the $i$-th vertex, $\X_{hi}=\X_h(\s_i)$, and $\F_h=\Grad_s\X_h$.
\label{pb:pbh}
\end{problem}

\begin{remark}
The evaluation of $\FF_h$ in Problem~\ref{pb:pbh} depends on the properties of
the solid body and on the approximating spaces. In particular, since we
assumed that $\X_h$ is piecewise linear, than the deformation gradient $\F_h$ is
piecewise constant and so is the approximate Piola--Kirchhoff tensor
$\P(\F_h)$. Hence the source term $\FF_h$ can be computed as follows:
\begin{equation}
\label{eq:discrforce}
\aligned
\langle\FF_h(t),\v\rangle&=
-\sum_{k=1}^{M_e}\int_{T_k}\PP|_{T_k}\colon\Grad_s\v(\X_h)\,\ds\\
&=-\sum_{k=1}^{M_e}\int_{\partial T_k}\PP|_{T_k}\N\cdot\v(\X_h)\dA\\
&=-\sum_{e\in\Eh}\int_e\jump{\PP}\cdot\v(\X_h)\dA\qquad\forall\v\in V_h,
\endaligned
\end{equation}
where $\jump{\PP}$, the jump of $\PP$ across the interelement edge $e$,
is defined as:
\begin{equation}
\jump{\PP}=\PP^+\N^+ +\PP^-\N^-,
\label{eq:jump}
\end{equation}
and $\N^+$ and $\N^-$ are the normals to the interface $e$
pointing outwards from the ``$+$'' or ``$-$'' element respectively. For more
details, the interested reader is referred to~\cite{BCG2011}.
\end{remark}

A fully discretized scheme for the approximation of Problem~\ref{pb:ibmvar}
has been introduced in~\cite{BCG2011}. If we denote by $\dt$ the time step
size and by $N$ the number of time steps,
then the approximation of the term $\d$ can be performed as follows:
\[
\langle\d_h^{n+1},\v\rangle=-\dr\int_\B\frac{\X_h^{n+1}-2\X_h^n+\X_h^{n-1}}{\dt^2}
\v(\X_h^n),
\]
where, as usual, the superscript $n$ refers to discrete quantities evaluated
after $n$ time steps.

For $n=0,\dots,N-1$ the solution strategy can then be summarized as follows.

{\bf Step 1.} Compute the source term $\FF^{n+1}_h$.
\begin{equation}
\label{eq:fsemi}
\langle\FF_h^{n+1},\v\rangle=
-\int_\B\P(\F_h^n(\s,t)):\Grad_s\v(\X^n_h(\s,t))\,d\s
\qquad\forall\v\in\Vh.
\end{equation}

{\bf Step 2.} Solve the Navier--Stokes equations:
find $(\u_h^{n+1},p_h^{n+1})\in \Vh\times\Qh$ such that
\begin{equation}
\label{eq:NSsemi}
\aligned
&\rho_f\left(\left(\frac{\u_h^{n+1}-\u_h^n}{\dt},\v\right)+
\left((\u_h^{n}\cdot\Grad\u_h^{n+1},\v)-(\u_h^{n}\cdot\Grad\v,\u_h^{n+1})\right)/2
\right)\\
&\quad+\nu(\Grads\u_h^{n+1},\Grads\v)
-(\div \v,p_h^{n+1})=&&\\
&\quad\displaystyle
-\dr\int_\B
\frac{\u_h^{n+1}(\X_h^n(s))-\u_h^n(\X_h^{n-1}(s))}{\dt}
\cdot\v(\X^{n}_h(s))\ds+\langle\FF_h^{n+1},\v\rangle&&
\forall \v\in\Vh\\[10pt]
&(\div\u_h^{n+1},q)=0&&\forall q\in\Qh.\\
\endaligned
\end{equation}

{\bf Step 3.} Advance the position of the structure (pointwise):
\begin{equation}
\label{eq:odesemi}
\frac{\X_{hi}^{n+1}-\X_{hi}^{n}}{\dt}=\u_h^{n+1}(\X_{hi}^{n})
\quad\forall i=1,\ldots,M.
\end{equation}

We observe that in Step~2 the unknown function $\u_h^{n+1}(\X_h^n(s))$
appears in the integral on the right hand side. By interpolating the basis
functions along the structure at time $t_n$, this gives an additional linear
contribution to the resulting algebraic system.

\section{The finite element Immersed Boundary Method with distributed Lagrange
multiplier}
\label{se:DLM}
In this section we describe the modification of the \feibm introduced
in~\cite{nuovodlm}. The main idea behind the new formulation consists in a
different treatment of the interaction between the fluid and solid velocities.
Namely, in Problem~\ref{pb:ibmvar} the motion of the
structure is designed by~\eqref{eq:velocity}, which for each $\s\in\B$ provides
an ordinary differential equation. As a consequence in the fully discretized
scheme we used~\eqref{eq:odesemi} to update the position of each point of the
structure. This gives rise to some restrictions on the choice of the
discretization parameters as it will be shown in the next section. In this
section we write~\eqref{eq:velocity} in variational form, so that we can have
more flexibility in the choice of the finite elements to be used.

Let us introduce three functional spaces $\LLambda$, $\Acca_1$ and $\Acca_2$
and two bilinear forms $\c_1:\LLambda\times\Acca_1\to\RE$ and
$\c_2:\LLambda\times\Acca_2\to\RE$ 
such that if $\v\in\Acca_1$ and $\Y\in\Acca_2$ are such that, for some given
$\overline\X$,
\[
\c_1(\mmu,\v(\overline\X))-\c_2(\mmu,\Y)=0\quad\forall\mmu\in\LLambda
\]
then $\v(\overline\X)=\Y$.

Then equation~\eqref{eq:velocity} can be written
\[
\c_1(\mmu,\u(\X(\cdot,t),t))-
\c_2\left(\mmu,\frac{\partial\X}{\partial t}(t)\right)=0
\quad\forall\mmu\in\LLambda,
\]
and can be interpreted as a constraint on our system. 
Therefore we introduce a Lagrange multiplier associated to such constraint
and split the first equation in Problem~\ref{pb:ibmvar} into two separate
equations, thus leading to the following \dlmibm version of the problem.

\begin{problem}
\label{pb:DLM}
Given $\u_0\in\Huo^d$ and $\X_0\in W^{1,\infty}(\B)$,
find $\u(t),p(t))\in\Huo^d\times\Ldo$, $\X(t)\in\Hub$, and
$\llambda(t)\in\LLambda$,
such that for almost every $t\in]0,T[$ it holds
\begin{equation}
\label{eq:DLM}
\aligned
  &\rho_f\left(\frac d {dt}(\u(t),\v)
  +\left((\u(t)\cdot\Grad\u(t),\v)-(\u(t)\cdot\Grad\v,\u(t))\right)/2\right)\\
  &\qquad+\nu(\Grads\u(t),\Grads\v)-(\div\v,p(t))\\
  &\qquad+\c_1(\llambda(t),\vcX)=0
   &&\ \forall\v\in\Huo^d
    \\
  &(\div\u(t),q)=0&&\ \forall q\in\Ldo
     \\
  &\dr\left(\frac{\partial^2\X}{\partial t^2}(t),\Y\right)_{\B}+
(\P(\F(t)),\Grad_s\Y)_{\B}-\c_2(\llambda(t),\Y)=0&&\ \forall\Y\in\Hub
     \\
  &\c_1(\mmu,\ucX)-
    \c_2\left(\mmu,\frac{\partial\X}{\partial t}(t)\right)
   =0 &&\ \forall\mmu\in\LLambda
     \\
  &\u(0)=\u_0\quad\mbox{\rm in }\Omega,\qquad\X(0)=\X_0\quad\mbox{\rm in }\B.
     \\
\endaligned
\end{equation}
\end{problem}
The finite element discretization of Problem~\ref{pb:DLM} is straightforward. 
In addition to the finite element spaces introduced in Sect.~\ref{se:fe-ibm},
we consider a finite element space $\LLambda_h\subseteq\LLambda$, so that we have
the following semidiscrete problem.
\begin{problem}
\label{pb:DLMh}
Given $\u_{h0}\in\Vh$ and $\X_{h0}\in\Sh$,
find $(\u_h(t),p_h(t))\in\Vh\times\Qh$, $\X_h(t)\in\Sh$, and
$\llambda_h(t)\in\Lh$,
such that for almost every $t\in]0,T[$ it holds
\begin{equation}
\label{eq:DLMh}
\aligned
  &\rho_f\left(\frac d {dt}(\u_h(t),\v)
  +\left((\u_h(t)\cdot\Grad\u_h(t),\v)-
   (\u(t)\cdot\Grad\v,\u_h(t))\right)/2\right)\\
  &\qquad+\nu(\Grads\u_h(t),\Grads\v)-(\div\v,p_h(t))\\
  &\qquad+\c_1(\llambda_h(t),\vcX)=0
   &&\ \forall\v\in\Vh
    \\
  &(\div\u_h(t),q)=0&&\ \forall q\in\Qh
     \\
  &\dr\left(\frac{\partial^2\X_h}{\partial t^2}(t),\Y\right)_{\B}+
(\P(\F_h(t)),\Grad_s\Y)_{\B}-\c_2(\llambda_h(t),\Y)=0&&\ \forall\Y\in\Sh
     \\
  &\c_1(\mmu,\uhcX)-\c_2\left(\mmu,\frac{\partial\X_h}{\partial t}(t)\right)
   =0 &&\ \forall\mmu\in\Lh
     \\
  &\u_h(0)=\u_{h0}\quad\mbox{\rm in }\Omega,
\qquad\X_h(0)=\X_{h0}\quad\mbox{\rm in }\B.
     \\
\endaligned
\end{equation}
\end{problem}
\begin{remark}
We observe that in this new formulation we do not need to evaluate the terms
$\d$ and $\F$, but the third equation represents the elasticity equation
with respect to the position of the body.  
\end{remark}
Let us introduce now the time discretization based on the Euler scheme. As in
the previous section, when computing along the structure terms involving
functions in $V_h$, we use the value of $\X$ at the previous time step, so that
we have the following scheme.
\begin{problem}
\label{pb:DLMhdt}
Given $\u_{h0}\in\Vh$ and $\X_{h0}\in\Sh$, suppose $\X_h^1\in\Sh$ has been
assigned (it can be computed formally by assuming $\X_h^{-1}=\mathbf{0}$ in
the scheme we are going to present).
Find $(\u_h^{n+1},p_h^{n+1})\in\Vh\times\Qh$, $\X_h^{n+1}\in\Sh$, and
$\llambda_h^{n+1}\in\Lh$, such that for all $n=1,\dots,N-1$ it holds
\begin{equation}
\label{eq:DLMhdt}
\aligned
  &\rho_f\left(\left(\frac{\u_h^{n+1}-\u_h^n}{\dt},\v\right)
  +\left((\u_h^n\cdot\Grad\u_h^{n+1},\v)-
   (\u_h^n\cdot\Grad\v,\u_h^{n+1})\right)/2\right)\\
  &\qquad+\nu(\Grads\u_h^{n+1},\Grads\v)-(\div\v,p_h^{n+1})\\
  &\qquad+\c_1(\llambda_h^{n+1},\vcXn)=0
   &&\ \forall\v\in\Vh
    \\
  &(\div\u_h^{n+1},q)=0&&\ \forall q\in\Qh
     \\
  &\dr\left(\frac{\X_h^{n+1}-2\X_h^n+\X_h^{n-1}}{\dt^2},\Y\right)_{\B}+
(\P(\F_h^{n+1}),\Grad_s\Y)_{\B}\\
  &\qquad-\c_2(\llambda_h^{n+1},\Y)=0&&\ \forall\Y\in\Sh
     \\
  &\c_1(\mmu,\uhcXn)-\c_2\left(\mmu,\frac{\X_h^{n+1}-\X_h^n}{\dt}(t)\right)
   =0 &&\ \forall\mmu\in\Lh
     \\
  &\u_h^0=\u_{h0}\quad\mbox{\rm in }\Omega,
\qquad\X_h^0=\X_{h0}\quad\mbox{\rm in }\B.
     \\
\endaligned
\end{equation}
\end{problem}
Problem~\ref{pb:DLMhdt} can be interpreted as a monolithic discretization of the
fluid-structure problem and, in the case of a linear model for the
Piola--Kirchhoff tensor $\P(\F)=\kappa\F=\kappa\Grad_s\X$,
has the following matrix structure:
\[
\left(
\begin{array}{cc|c|c}
\mathsf{A}&\mathsf{B}^\top&0&\mathsf{L}_f(\X_h^n)^\top\\
\mathsf{B}&0&0&0\\
\hline\\[-10pt]
0&0&\mathsf{A}_s&-\mathsf{L}^\top_s\\
\hline\\[-10pt]
\mathsf{L}_f(\X_h^n)&0&-\mathsf{L}_s&0
\end{array}
\right)
\left(
\begin{array}{c}
\u_h^{n+1}\\p_h^{n+1}\\
\hline\\[-10pt]
\X_h^{n+1} \\
\hline\\[-10pt]
\lh^{n+1}
\end{array}
\right)=
\left(
\begin{array}{c}
\mathsf{f}\\0\\
\hline\\[-10pt]
\mathsf{g}\\
\hline\\[-10pt]
\mathsf{d}
\end{array}
\right)
\]
where, denoting by $\pphi$, $\psi$, $\cchi$ and
$\zzeta$ the basis functions
respectively in $\Vh$, $\Qh$, $\Sh$ and $\Lh$, we have used the following
notation:
\[
\aligned
&\mathsf{A}=\frac{\rho_f}{\dt}\mathsf{M}_f+\mathsf{K}_f
\quad\text{with }(\mathsf{M}_f)_{ij}=(\pphi_j,\pphi_i),
\ (\mathsf{K}_f)_{ij}=a(\pphi_j,\pphi_i)+b(\u_h^n,\pphi_j,\pphi_i)\\
&\mathsf{B}_{ki}=-(\div\pphi_i,\psi_k)\\
&\mathsf{A}_s=\frac{\dr}{\dt^2}\mathsf{M}_s+\mathsf{K}_s
\quad\text{with }(\mathsf{M}_s)_{ij}=(\cchi_j,\cchi_i)_{\B},
\ (\mathsf{K}_s)_{ij}=\kappa(\Grad_s\cchi_j,\Grad_s\cchi_i)_{\B}\\
&(\mathsf{L}_f(\X_h^n))_{lj}=\c_1(\zzeta_l,\pphi_j(\X_h^n))\\
&(\mathsf{L}_s)_{lj}=\c_2(\zzeta_l,\cchi_j)\\
&\mathsf{f}_i=\frac{\rho_f}{\dt}(\mathsf{M}_f\u_h^n)_i\\
&\mathsf{g}_i=\frac{\dr}{\dt^2}\left(\mathsf{M}_s(2\X_h^n-\X_h^{n-1})\right)_i\\
&\mathsf{d}_l=-\frac1{\dt}(\mathsf{L}_s\X_h^n)_l.
\endaligned
\]
We observe that this system can be solved with a block iterative procedure
which allows for the use of Navier--Stokes and elasticity solvers. 

The block structure of the matrix highlights the fact that at each time step
we have to solve a saddle point problem whose analysis will be the object of a
forthcoming paper.

\section{Stability analysis}
\label{se:stab}
In this section we report the stability estimates for the \ibm
formulations of the fluid-structure interaction problem that we have
introduced in the previous sections. The main results concerns the stability in time
of the fully discrete schemes. We shall see that the \dlmibm method is superior
to the \feibm from this point of view since it is unconditionally stable, while
the \feibm scheme requires that the discretization parameters are chosen in a
appropriate way.

Since the solid is composed by a hyperelastic material, it is
characterized by a positive energy density
$W(\F)$ which depends only on the deformation gradient and
the first Piola--Kirchhoff stress tensor can be obtained by derivation with
respect to deformation gradient
$\P(\F(\s,t))=\frac{\partial W }{\partial \F}(\F(\s,t))$.
We assume that $W$ is a $C^1$ convex function over the set of second order
tensors. Moreover, the elastic potential energy of the body is given by:
\begin{equation}
\label{eq:potenergy}
E\left(\X(t)\right)=\int_\B W(\F(s,t))\ds.
\end{equation}
It is not difficult to show the following energy estimate for 
the continuous versions of \feibm and \dlmibm (see~\cite{BCG2011,nuovodlm}).
\begin{proposition}
\label{pr:stab-cont}
Let $\u(t)\in(\Huo)^d$, $p(t)\in\Ldo$ and $\X(t)\in\Hub$ be solution either of
Problem~\ref{pb:ibmvar} or Problem~\ref{pb:DLM}, then the following estimate
holds true
\begin{equation}
\label{eq:energyest}
\frac{\rho_f}2\frac{d}{dt}||\u(t)||^2_0+\nu||\Grads\u(t)||^2_0+
\frac{\dr}2\frac{d}{dt}\left\|\frac{\partial \X}{\partial t}\right\|^2_{0,\B}
+\frac{d}{dt}E(\X(t))=0,
\end{equation}
where $\|\cdot\|_0$ and $\|\cdot\|_{0,\B}$
denote the norms in $L^2(\Omega)$ and $L^2(\B)$, respectively.
\end{proposition}
\begin{proof}
Let us consider first Problem~\ref{pb:ibmvar}. We take $\v=\u(t)$
in the first equation of~\eqref{eq:ibmvar}, obtaining
\[
\frac{\rho_f}2\frac{d}{dt}||\u(t)||^2_0+\nu||\Grads\u(t)||^2_0=
\langle\d(t),\u\rangle+\langle\FF(t),\u\rangle.
\]
Then the
definitions of the source terms $\d$ and $\FF$ give:
\[
\aligned
\langle\d(t),\u\rangle&=
-\dr\int_\B\frac{\partial^2\X}{\partial t^2}\u(\X(\s,t))\ds
=-\dr\int_\B\frac{\partial^2\X}{\partial t^2}\frac{\partial\X}{\partial t}\ds\\
&=-\frac{\dr}2\frac{d}{dt}\int_\B\left(\frac{\partial\X}{\partial t}\right)^2\ds
=-\frac{\dr}2\frac{d}{dt}\left\|\frac{\partial\X}{\partial t}\right\|_{0,\B}^2
\\
\langle\FF(t),\u\rangle&=-\int_\B\P(\F(\s,t)):\Grad_s\u(\X(\s,t))\ds
=-\int_\B\P(\F(\s,t)):\Grad_s\frac{\partial\X}{\partial t}\ds\\
&=-\int_\B\P(\F(\s,t)):\frac{\partial\Grad_s\X}{\partial t}\ds=
-\int_\B\frac{\partial W}{\partial\F}(\F(\s,t)):\frac{\partial\F}{\partial t}\ds
\\
&=-\frac{d}{dt}\int_\B W(\F(\s,t))\ds=-\frac{d}{dt}E\left(\X(t)\right),
\endaligned
\]
which concludes the proof. 

For the stability of the solution of Problem~\ref{pb:DLM} the proof is even
simpler. It is enough to take as test functions $\v=\u(t)$, $q=p(t)$,
$\Y=\partial\X(t)\slash\partial t$, and $\mmu=\llambda(t)$ in the variational
equations in~\eqref{eq:DLM} and the result is achieved by summing up the
equations and using the same computation as before to deal with the term
containing the Piola--Kirchhoff stress tensor.
\end{proof}
The stability property of the semidiscrete problems can be obtained with the
same arguments as in Proposition~\ref{pr:stab-cont}.
\begin{proposition}
\label{pr:stab-semidiscr}
Let $\u_h(t)\in\Vh$, $p_h(t)\in\Qh$ and $\X_h(t)\in\Sh$ be solution of
Problem~\ref{pb:pbh} or Problem~\ref{pb:DLMh}, then the following estimate
holds true
\begin{equation}
\label{eq:energyesth}
\frac{\rho_f}2\frac{d}{dt}||\u_h(t)||^2_0+\nu||\Grads\u_h(t)||^2_0+
\frac{\dr}2\frac{d}{dt}\left\|\frac{\partial\X_h}{\partial t}\right\|^2_{0,\B}
+\frac{d}{dt}E(\X_h(t))=0.
\end{equation}
\end{proposition}
When we consider the fully discrete schemes,
namely~\eqref{eq:fsemi}-\eqref{eq:odesemi} for \feibm and
Problem~\ref{pb:DLMhdt}, the situation changes and we have different type of
results. In the case of \feibm we state in the following proposition that the
energy estimate holds true provided a \cfl condition is satisfied.
\begin{proposition}
\label{pr:stab-feibm}
Assume that the energy density $W$ is a $C^1$ convex function.
Given $\u_{0h}\in\Vh$ and $\X_{h0}\in\Sh$, for $n=1,\dots,N$ let
$\u_h^n\in\V_h$, $p_h^n\in\Qh$
and $\X_h^n\in\Sh$ satisfy~\eqref{eq:fsemi}-\eqref{eq:odesemi}. Then the
following energy estimate holds true
\begin{equation}
\label{eq:estsemi}
\begin{split}
\frac{\rho_f}{2\dt}
&\left(\|\u_h^{n+1}\|^2_{0}-\|\u_h^{n}\|^2_{0} \right)
+(\nu + \nu_a) \|\Grad \u_h^{n+1} \|^2_{0}
+\frac{1}{\dt}\left( E\left(\X_h^{n+1}\right)-E\left(\X_h^{n}\right)\right)\\
&+\frac1{2\dt}\dr
\left(\left\|\frac{\X_h^{n+1}-\X_h^n}{\dt}\right\|^2_{0,\B}-
\left\|\frac{\X_h^{n}-\X_h^{n-1}}{\dt}\right\|^2_{0,\B}\right)
\le0
\end{split}
\end{equation}
where $\nu_a$ is given by
\begin{equation}
\label{eq:artificial-viscosity-fe-be}
\nu_a = -\kappa_{max}  Ch_s^{m-2}\,h_x^{-(d-1)}\dt L^nC_e^{n},
\end{equation}
$L^n$ is the maximum distance between any two consecutive
vertices of the Lagrangian mesh and  $C_e^n$ stands for the maximum
number of Lagrangian elements that touch the same Eulerian element at the
given time step. 
\end{proposition}
We summarize the \cfl condition for different values of fluid
and solid dimension in Table~\ref{tb:cfl}.
\begin{table}
\begin{center}
\begin{tabular}{c|c|l}
\hline
\rule{0pt}{12pt}
space dim.&solid dim.&\cfl condition\\
\hline%
\rule{0pt}{12pt}%
$2$&$1$&$L^n\Delta t\le Ch_x h_s$\\
$2$&$2$&$L^n\Delta t\le Ch_x$\\
$3$&$2$&$L^n\Delta t\le Ch_x^2$\\
$3$&$3$&$L^n\Delta t\le Ch_x^2/h_s$\\
\hline
\end{tabular}
\end{center}
\caption{\cfl condition according to fluid and structure dimensions.}
\label{tb:cfl}
\end{table}

The fully discrete \dlmibm scheme is unconditionally stable.
\begin{proposition}
\label{pr:stab-DLM}
Assume that the energy density $W$ is a $C^1$ convex function.
Let $\u^n\in(\Huo)^d$ and $\X^n\in\Hub$
for $n=0,\dots,N$ satisfy Problem~\ref{pb:DLMhdt} with
$\X^n\in W^{1,\infty}(\B)^d$, then
the following estimate holds true for all $n=0,\dots,N-1$
\begin{equation}
\label{eq:energy_sd}
\aligned
&\frac{\rho_f}{2\dt}\left(\|\u^{n+1}\|^2_0-\|\u^n\|^2_0\right)+
\nu\|\Grads\u^{n+1}\|^2_0\\
&+\frac{\dr}{2\dt}\left(\left\|\frac{\X^{n+1}-\X^n}{\dt}\right\|^2_{0,\B}
-\left\|\frac{\X^n-\X^{n-1}}{\dt}\right\|^2_{0,\B}\right)+
\frac{E(\X^{n+1})-E(\X^n)}{\dt} \le0.
\endaligned
\end{equation}
\end{proposition}

\section{Numerical tests}
\label{se:num}

In this section we report on some numerical tests that we have performed
during our research in the finite element approach to \ibm.
The aim of these tests is to show that the \feibm and \dlmibm methods can be
efficiently implemented and used for the approximation of fluid-structure
interaction problems.
In particular, the presented results have been already published in previous
papers or have never been published before even if they have been obtained as
the results of previous research. This is the case, for instance, of the
snapshots of the presented animations.

We start by showing some snapshots taken from a three dimensional simulation
involving the interaction of a codimension one closed solid surface and a
fluid contained in a cubic box.
The initial configuration is reported in Figure~\ref{fig:ellipsoid0} and
correspond to an ellipsoid stretched in one of the horizontal directions.
\begin{figure}
\includegraphics[width=8cm]{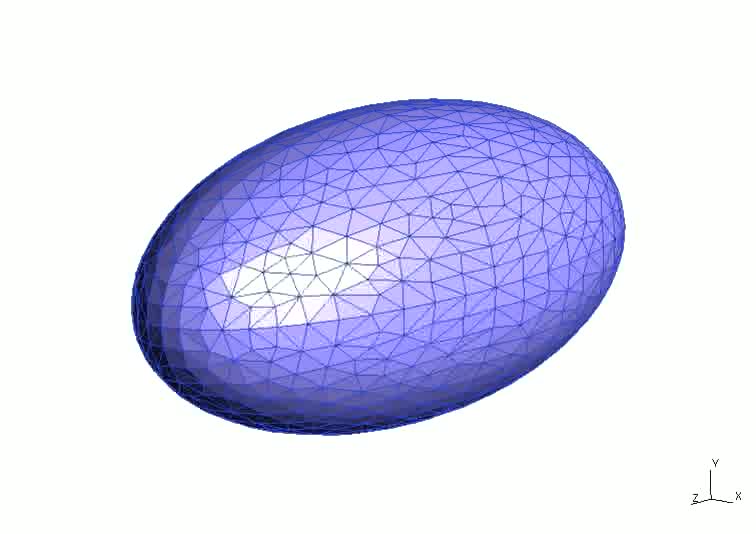}
\caption{\feibm simulation of an elastic ellipsoid: initial
configuration\label{fig:ellipsoid0}}
\end{figure}
The initial fluid is at rest, so that the system is driven only by the elastic
force of the solid which is tending to its spherical equilibrium
configuration. The evolution of the system is reported in
Figure~\ref{fig:ellipsoid}: as expected the solid tends to a sphere.

\begin{figure}
\includegraphics[width=5cm]{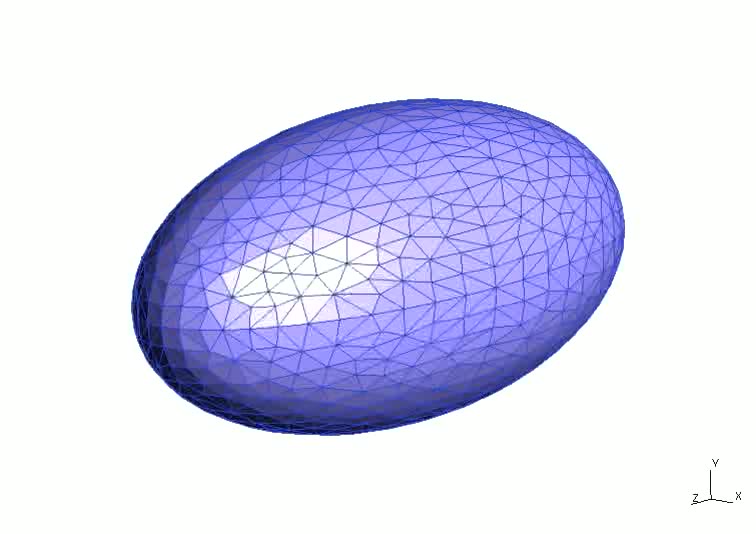}
\includegraphics[width=5cm]{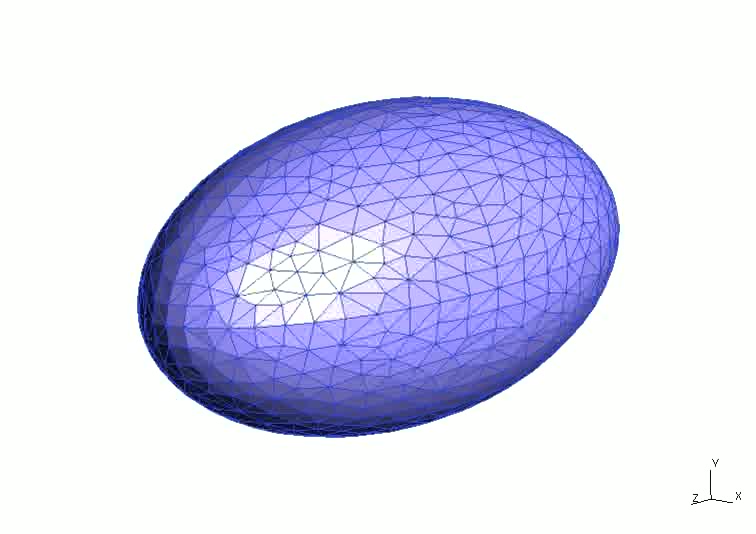}\\
\includegraphics[width=5cm]{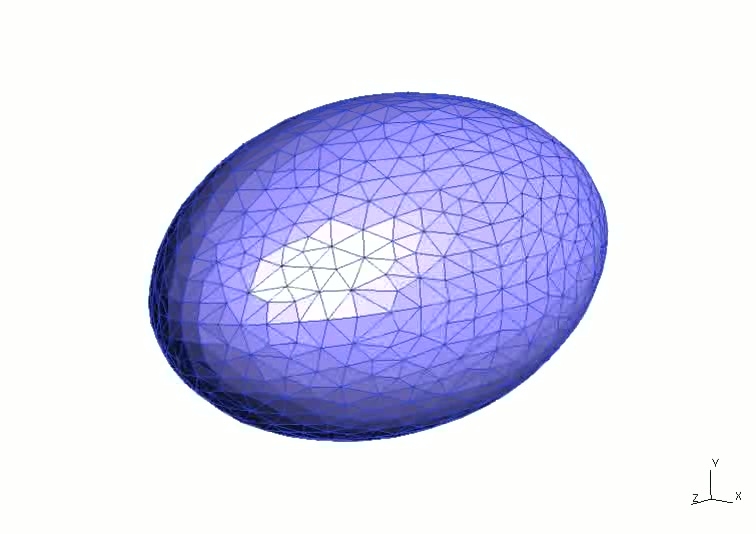}
\includegraphics[width=5cm]{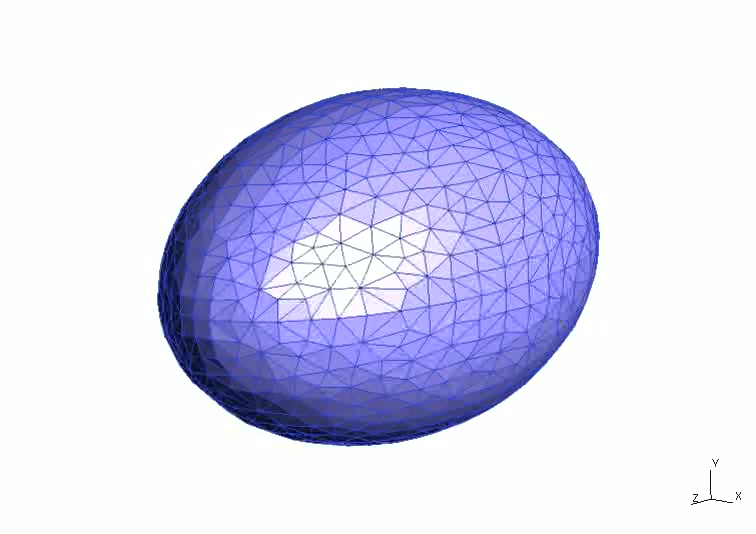}\\
\includegraphics[width=5cm]{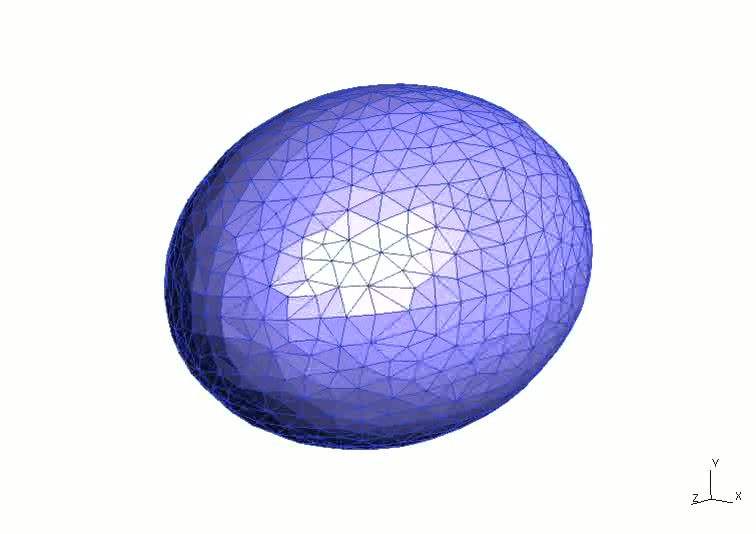}
\includegraphics[width=5cm]{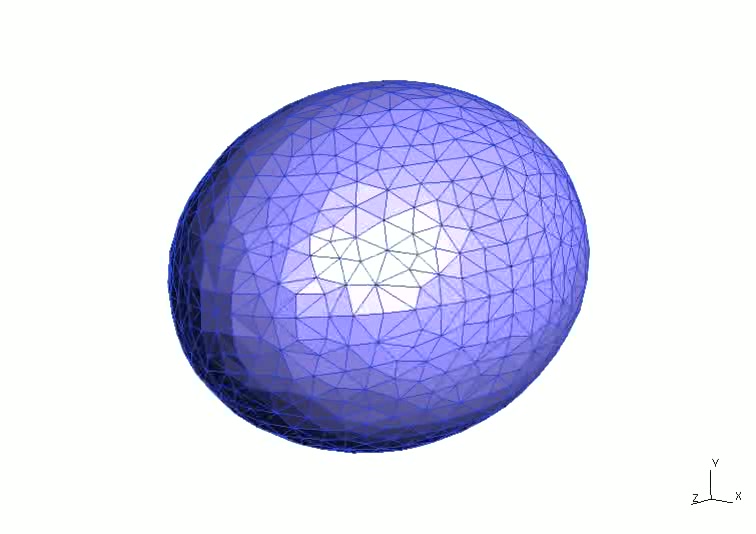}
\caption{\feibm simulation of an elastic ellipsoid: snapshots of the time
evolution (left-to-right and top-to-bottom)\label{fig:ellipsoid}}
\end{figure}

Our second simulation discusses the situation when more than one solid is
present. More precisely, the initial configuration is reported in
Figure~\ref{fig:tube0} where two circular structures of codimension one are
immersed in a fluid confined in a rectangular box.
\begin{figure}
\includegraphics[width=8cm]{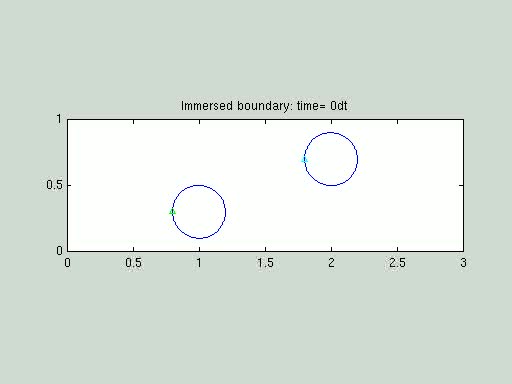}
\caption{\feibm simulation of two elastic bodies: initial 
configuration\label{fig:tube0}}
\end{figure}
In the top half of the box the fluid is moving from right to left, while in
the bottom part the fluid moves in the opposite direction.
Figure~\ref{fig:tube} shows the evolution of the system: it can be appreciated
that the two structures change their shape as a consequence of their
interaction and the incompressibility of the fluid.

\begin{figure}
\includegraphics[width=5cm]{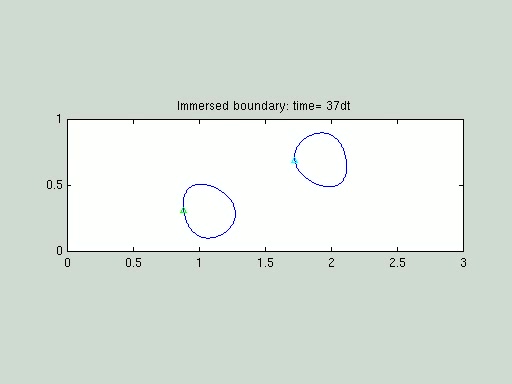}
\includegraphics[width=5cm]{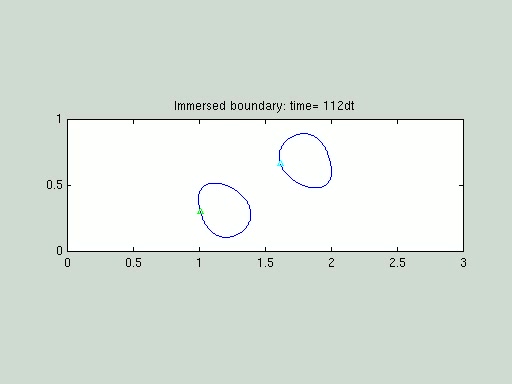}\\
\includegraphics[width=5cm]{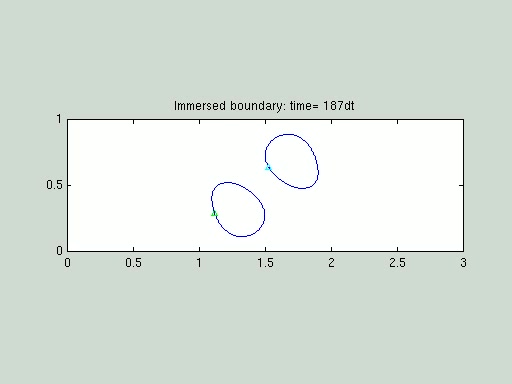}
\includegraphics[width=5cm]{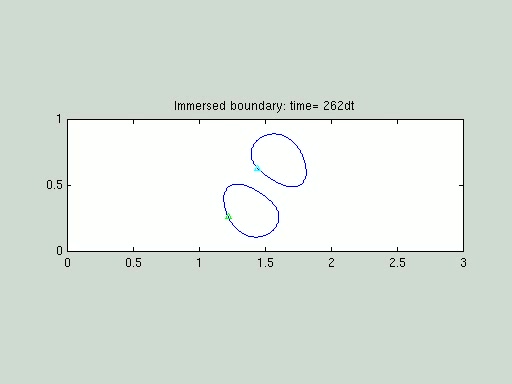}\\
\includegraphics[width=5cm]{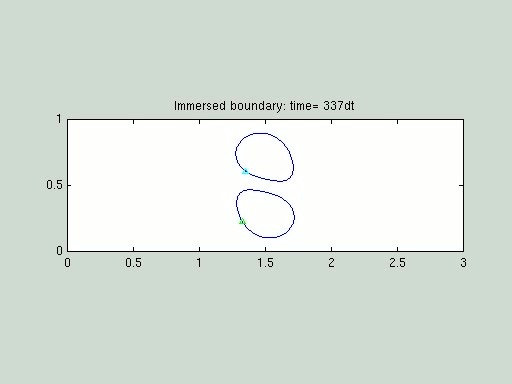}
\includegraphics[width=5cm]{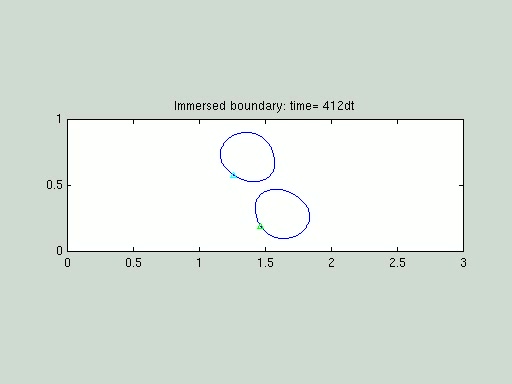}
\caption{\feibm simulation of two circular structures moving in opposite
directions: snapshots of the time evolution (left-to-right and
top-to-bottom)\label{fig:tube}}
\end{figure}

The next set of numerical tests discusses the mass conservation properties of
the \feibm in the spirit of~\cite{bcgg2012,bcggumi}. In this case it has been
shown that the mass conservation depends on the ability of the Stokes solver
to provide a good discretization of the divergence free condition. It is clear
that discontinuous pressure finite elements provide a better approximation of
the divergence free constraint. In Figure~\ref{fig:mass} we compare continuous
(solid) and discontinuous (dashed) pressure approximation for the \feibm. The
test is the two dimensional analogue of the one reported in
Figures~\ref{fig:ellipsoid0} and~\ref{fig:ellipsoid}: an elliptic elastic
string tends to a circular equilibrium configuration. The mass conservation
property is expressed by the conservation of the area inside the structure. We
use Hood--Taylor and $P_1$-iso-$P_2-P_0$ Stokes element for the continuous
pressure simulations (see~\cite{bbf}) and the corresponding enhanced elements
presented in~\cite{bcgg2012} for the discontinuous case.
\begin{figure}
\includegraphics[width=8cm]{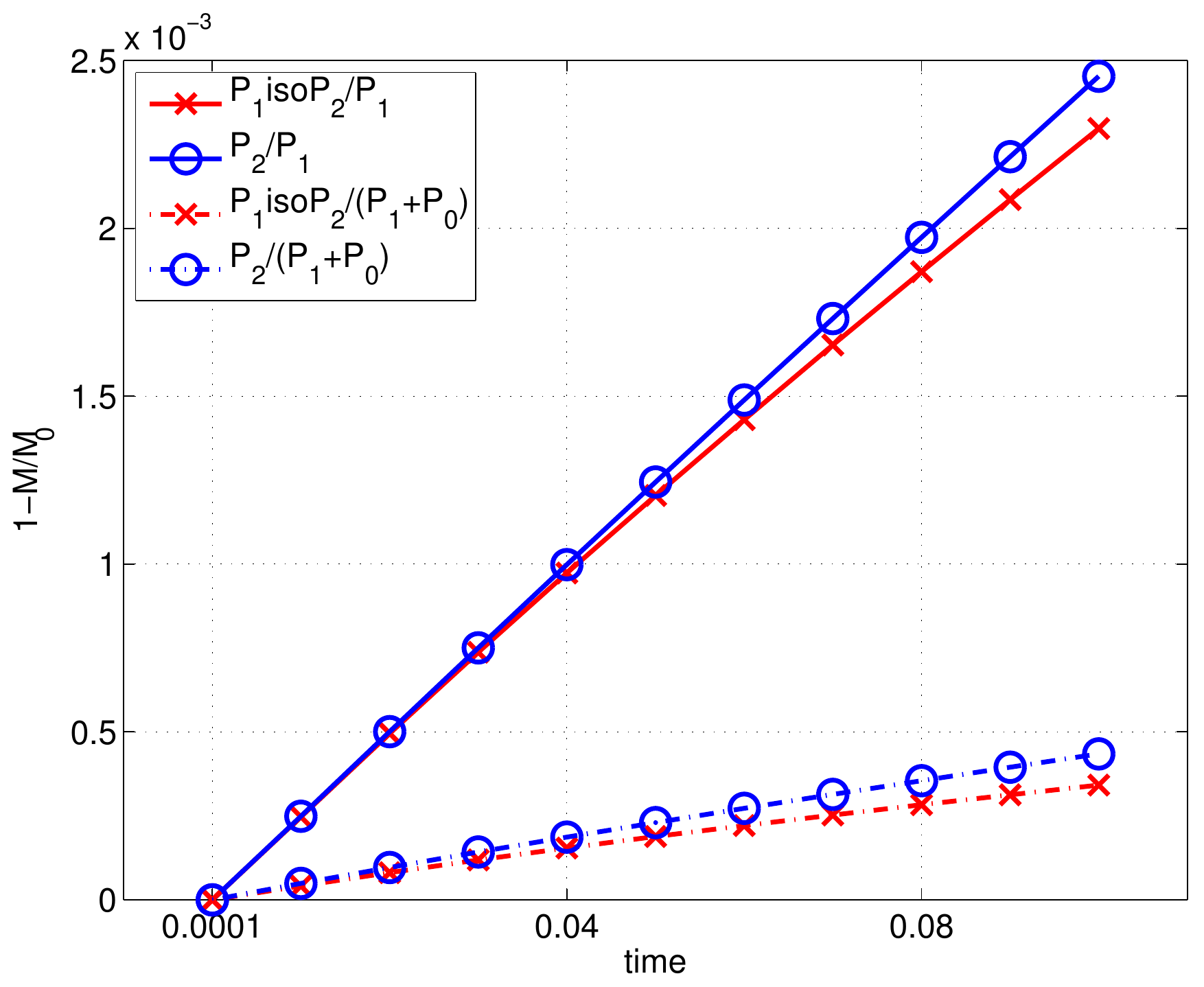}
\caption{Mass conservation of continuous versus discontinuous pressure finite
elements\label{fig:mass}}
\end{figure}
The slope of the lines corresponds to the rate of mass loss with respect to
time.

Figure~\ref{fg:p3p2} shows that in this respect computations obtained with
\dlmibm are better than the original \feibm. The mesh of the structure
(an ellipse tending to a circle) is very coarse in order to emphasize the
phenomenon and the used Stokes element is the enriched higher-order
Hood--Taylor $P_3-(P_2^c+P_1)$. This aspect will be investigated in our
future research.

\begin{figure}
\subcaptionbox{\feibm computation}
{\includegraphics[width=5cm]{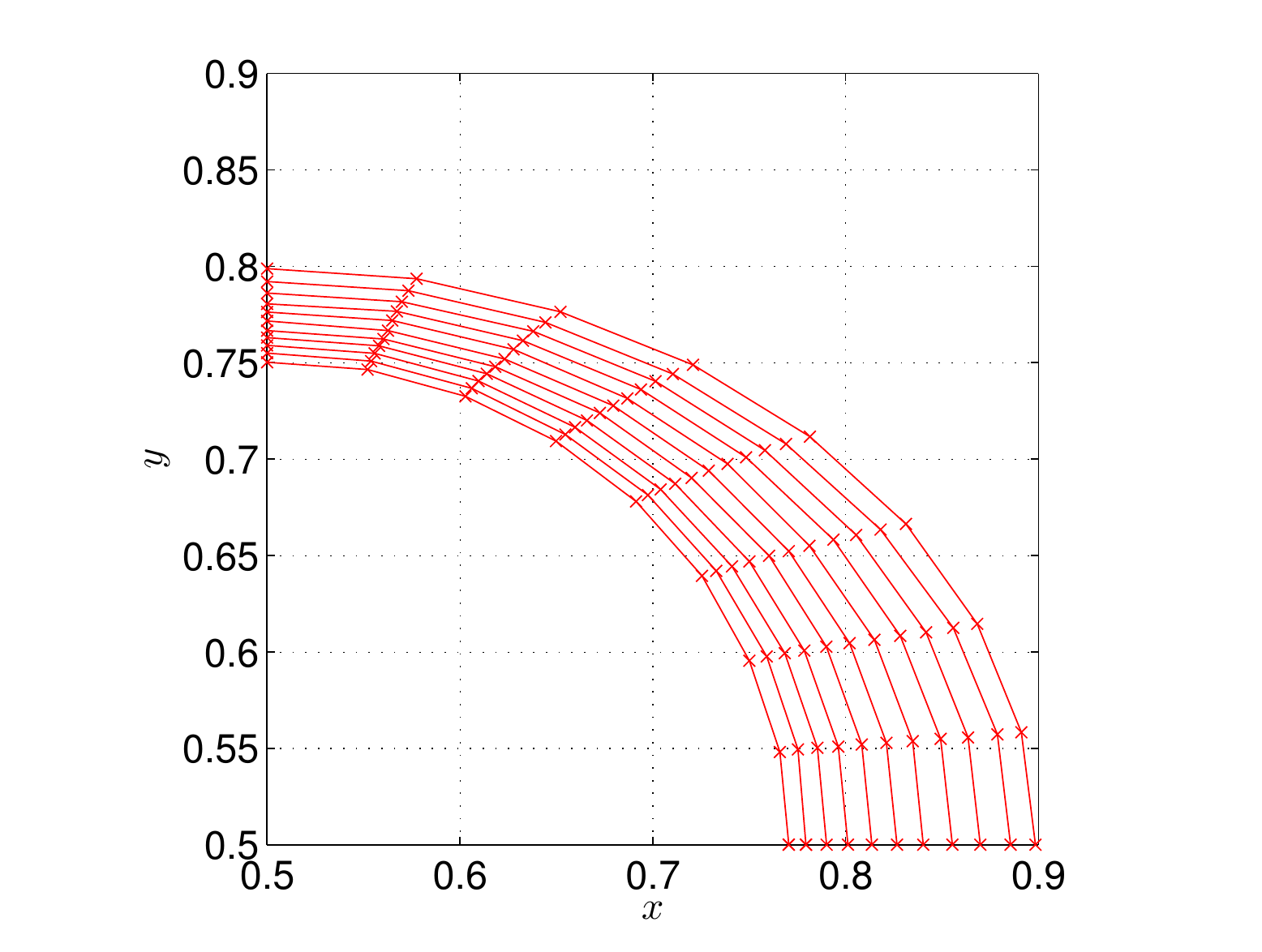}}
\subcaptionbox{\dlmibm computation}
{\includegraphics[width=5cm]{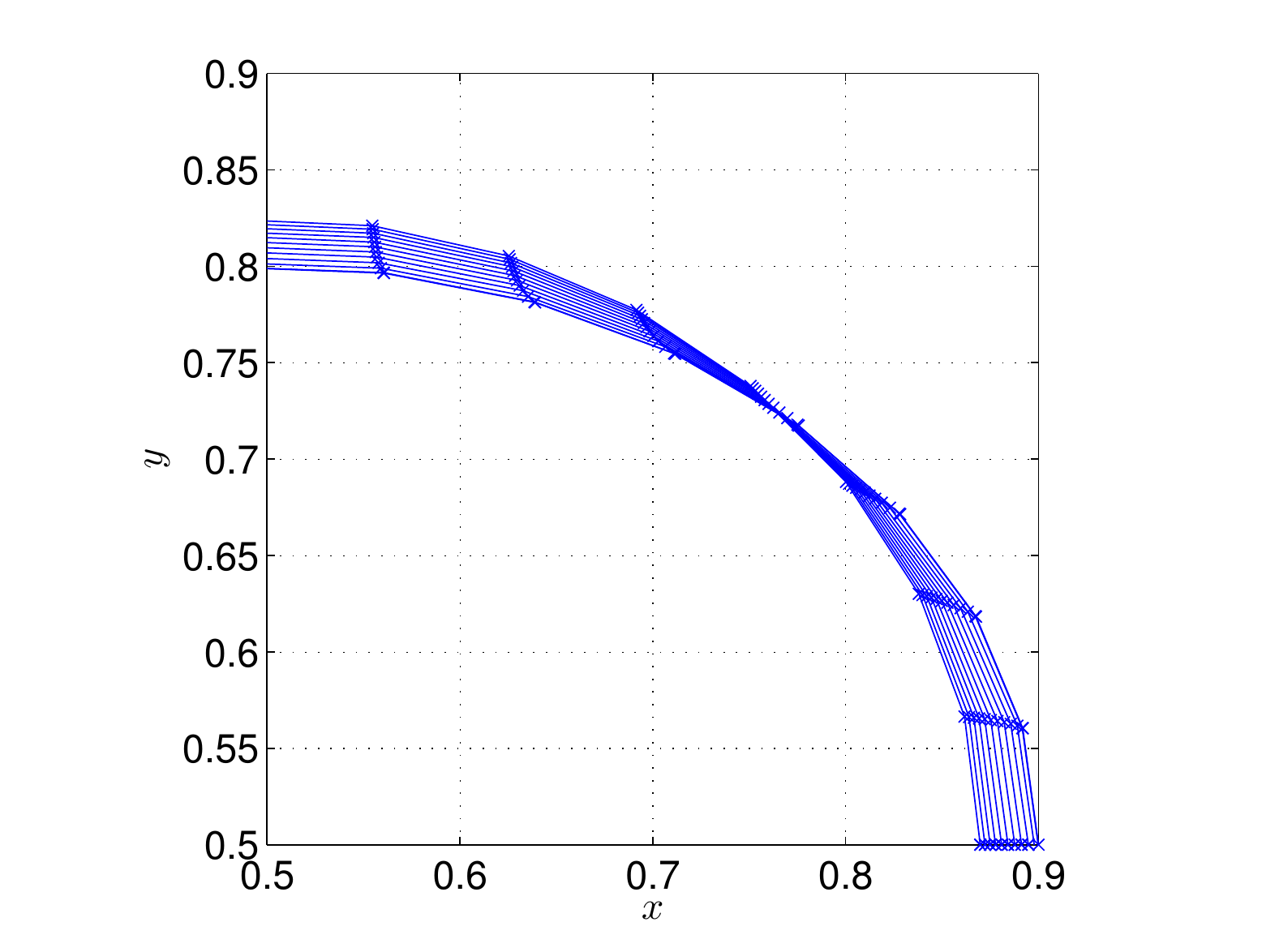}}
\caption{Mass conservation for \feibm and \dlmibm\label{fg:p3p2}}
\end{figure}

The last set of numerical experiments corresponds to the stability of the
fully discrete scheme. It is well known that non-implicit schemes usually
result in severe restrictions to the time step. For the ALE formulation of
fluid-structure interactions, it has been shown that this can produce a
unconditionally unstable method~\cite{causin}.
In Section~\ref{se:stab} we recalled that, on the other hand, the
semi-implicit formulation of our \feibm has been proved to be stable if a
suitable \cfl condition is satisfied. In Section~\ref{se:stab} it is also
recalled that the stability properties of the \dlmibm are even better: in this
case unconditional stability occurs, as it has been shown in~\cite{nuovodlm}.
In Figures~\ref{fig:thin_energy_rho03} and~\ref{fig:thick_energy_rho03} we
report the ratio of the total energy at time $n$ over the initial energy.
When the stability is violated, the ratio blows up; on the other hand the
energy remains bounded if the method is stable. 

 \begin{figure}
 \centering
 \subcaptionbox{$\Delta t = 10^{-1}$, $h_s = 1/8$.}
   {\includegraphics[width=4cm]{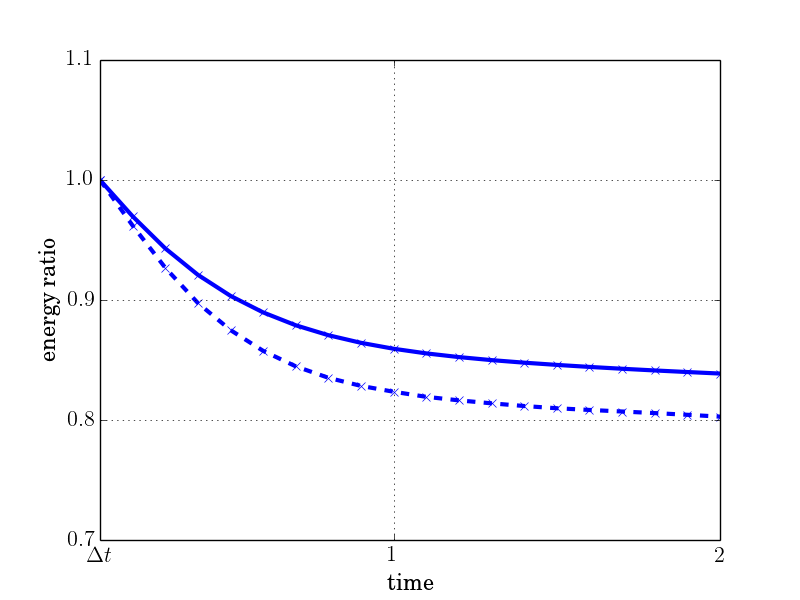}}
 \subcaptionbox{$\Delta t = 10^{-1}$, $h_s = 1/16$.}
   {\includegraphics[width=4cm]{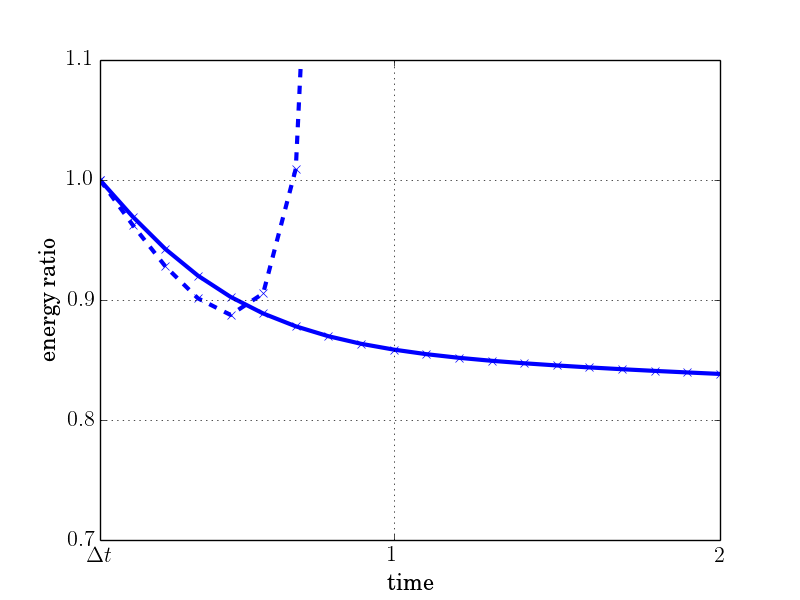}}
 \subcaptionbox{$\Delta t = 10^{-1}$, $h_s = 1/32$.}
   {\includegraphics[width=4cm]{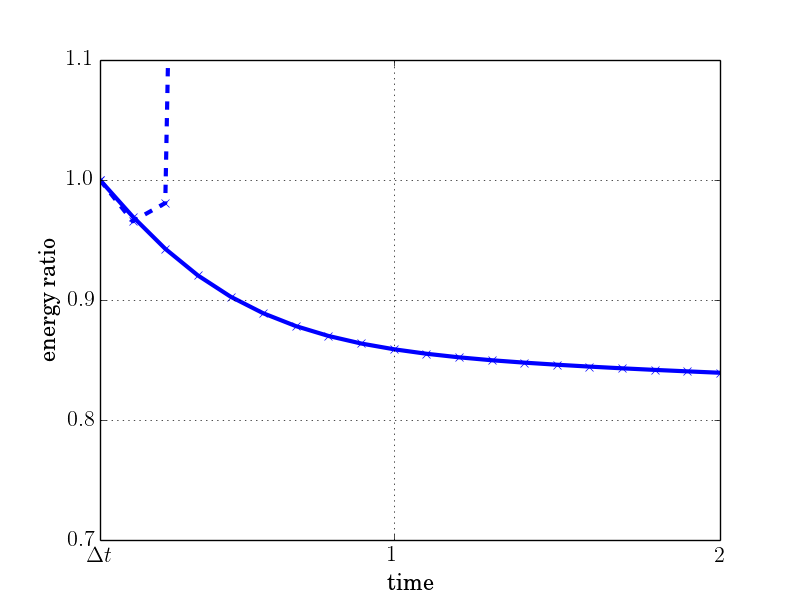}}\\
 \subcaptionbox{$\Delta t = 5\cdot 10^{-2}$, $h_s = 1/8$.}
   {\includegraphics[width=4cm]{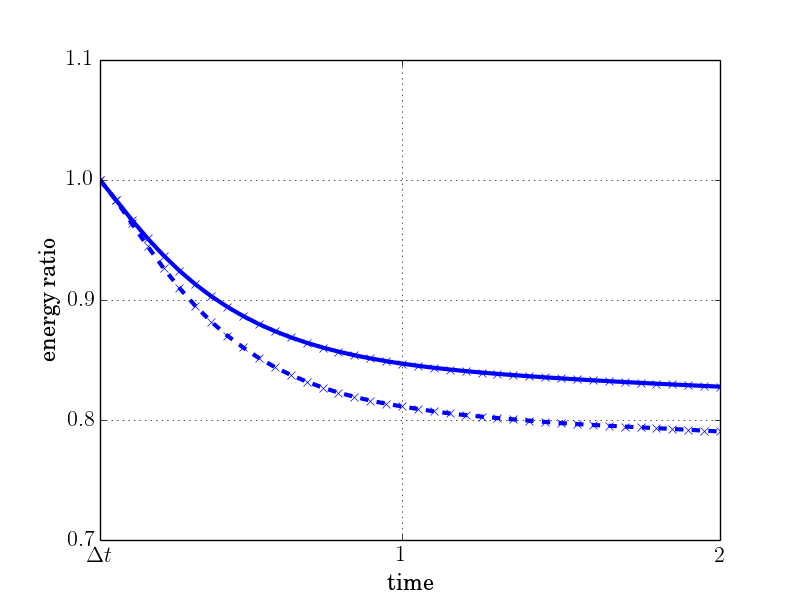}}
 \subcaptionbox{$\Delta t = 5\cdot 10^{-2}$, $h_s = 1/16$.}
   {\includegraphics[width=4cm]{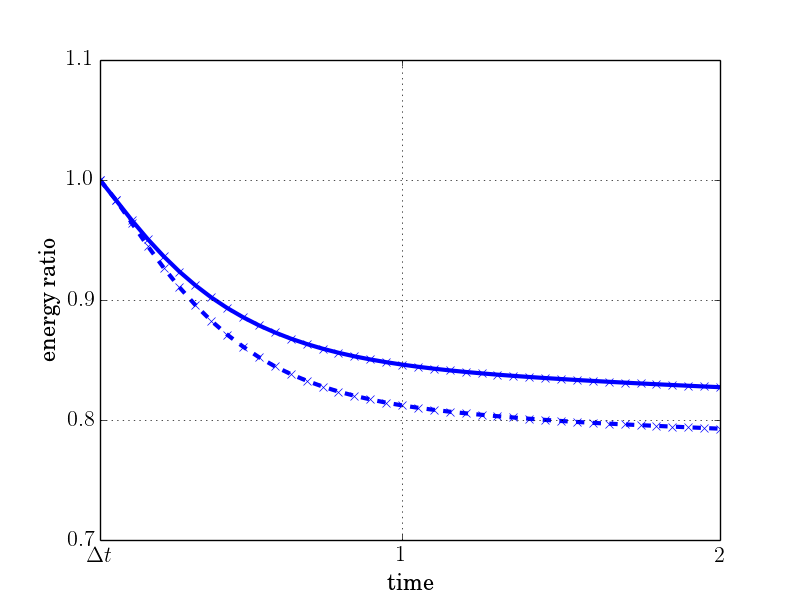}}
 \subcaptionbox{$\Delta t = 5\cdot 10^{-2}$, $h_s = 1/32$.}
   {\includegraphics[width=4cm]{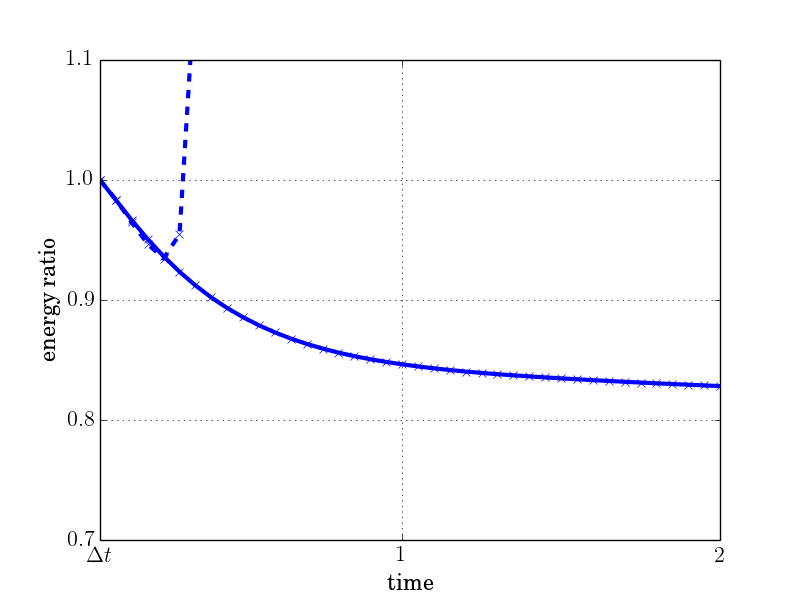}}
 \caption{Energy ratio for codimension one structure. The structure elastic
constant $\kappa = 5$, $h_x = 1/32$,
 the fluid viscosity $\nu = 1$, $\delta\rho = 0.3$. The solid line correspond
to the \dlmibm scheme, while 
 the dashed line marks the energy for the \feibm scheme.}
 \label{fig:thin_energy_rho03}
 \end{figure}

 \begin{figure}
 \centering
 \subcaptionbox{$\Delta t = 10^{-1}$, $h_x = 1/4$.}
   {\includegraphics[width=4cm]{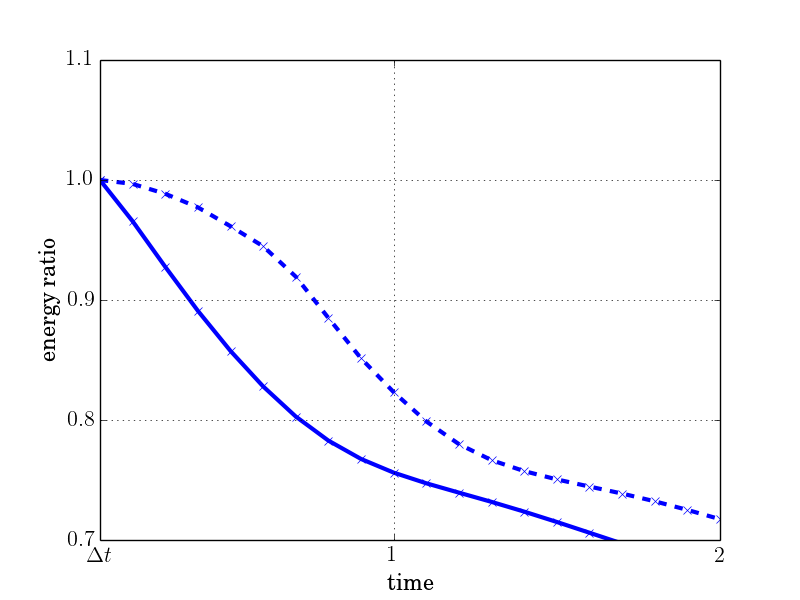}}
 \subcaptionbox{$\Delta t = 10^{-1}$, $h_x = 1/8$.}
   {\includegraphics[width=4cm]{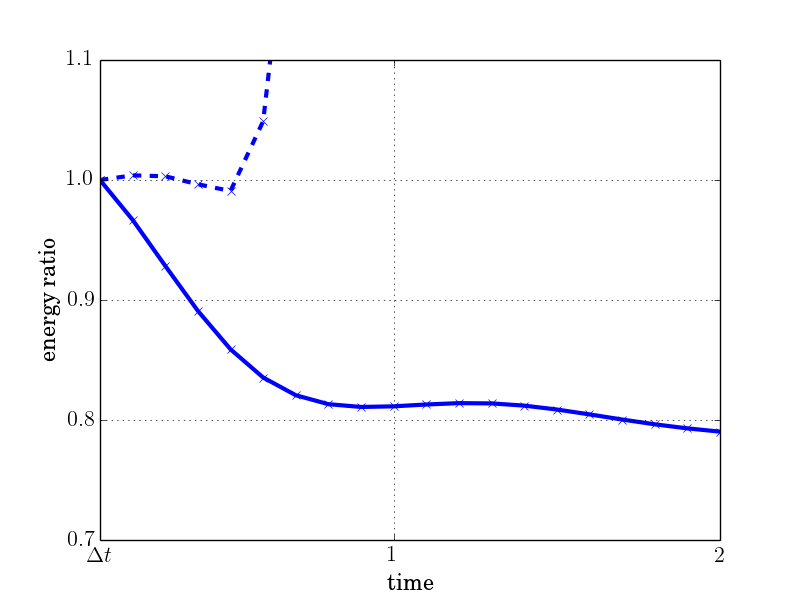}}
 \subcaptionbox{$\Delta t = 10^{-1}$, $h_x = 1/16$.}
   {\includegraphics[width=4cm]{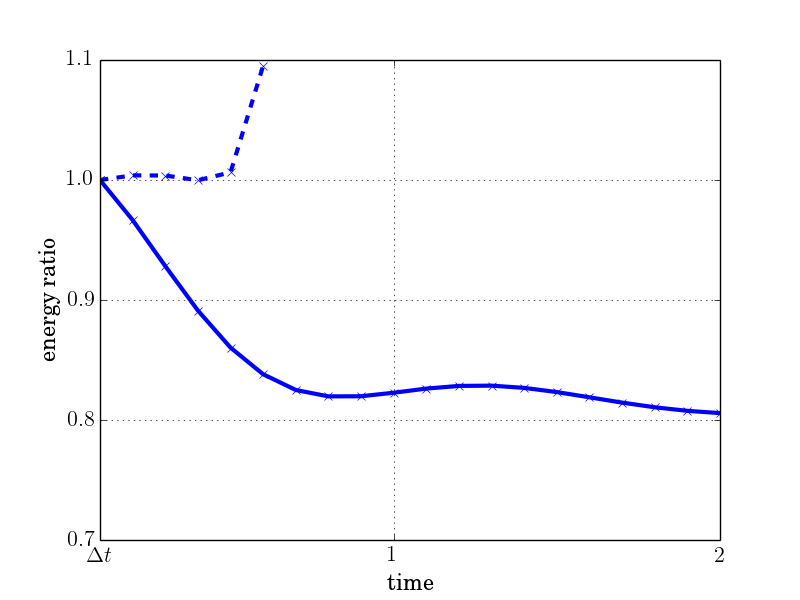}}\\
  \subcaptionbox{$\Delta t = 5\cdot10^{-2}$, $h_x = 1/4$.}
   {\includegraphics[width=4cm]{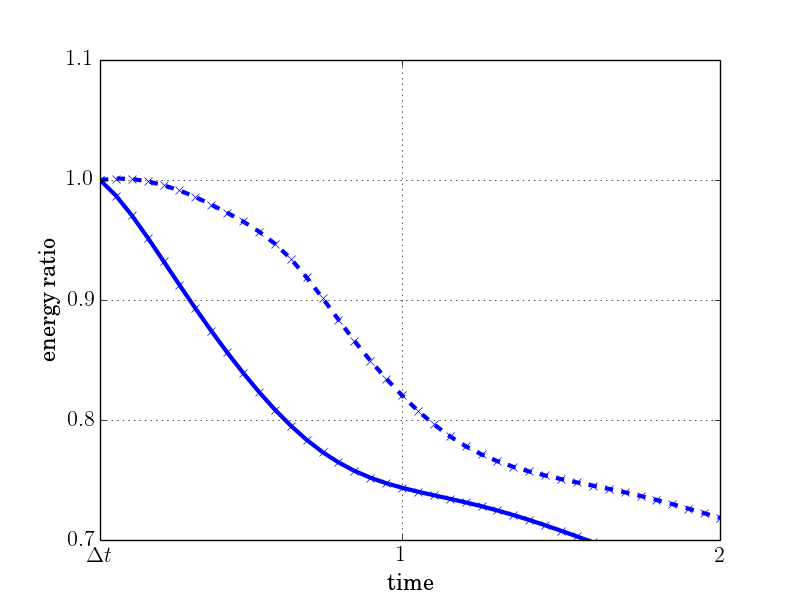}}
 \subcaptionbox{$\Delta t = 5\cdot10^{-2}$, $h_x = 1/8$.}
   {\includegraphics[width=4cm]{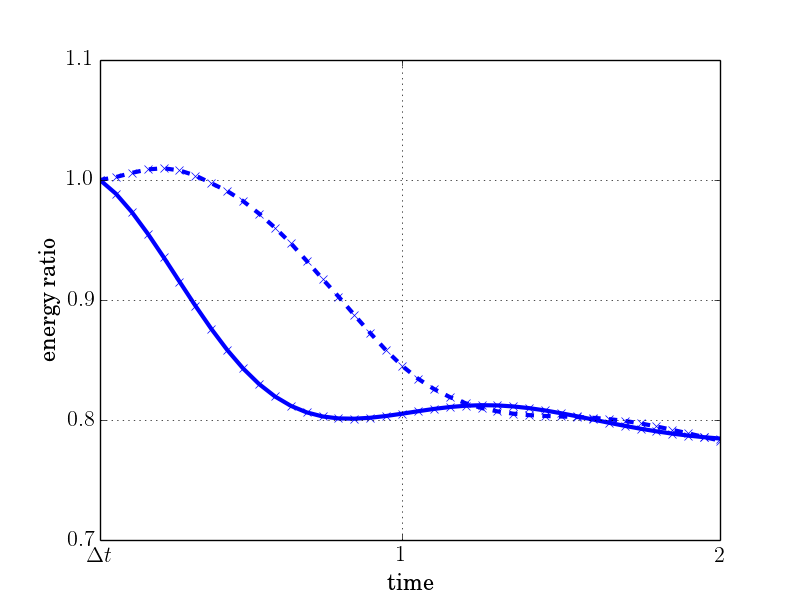}}
 \subcaptionbox{$\Delta t = 5\cdot10^{-2}$, $h_x = 1/16$.}
   {\includegraphics[width=4cm]{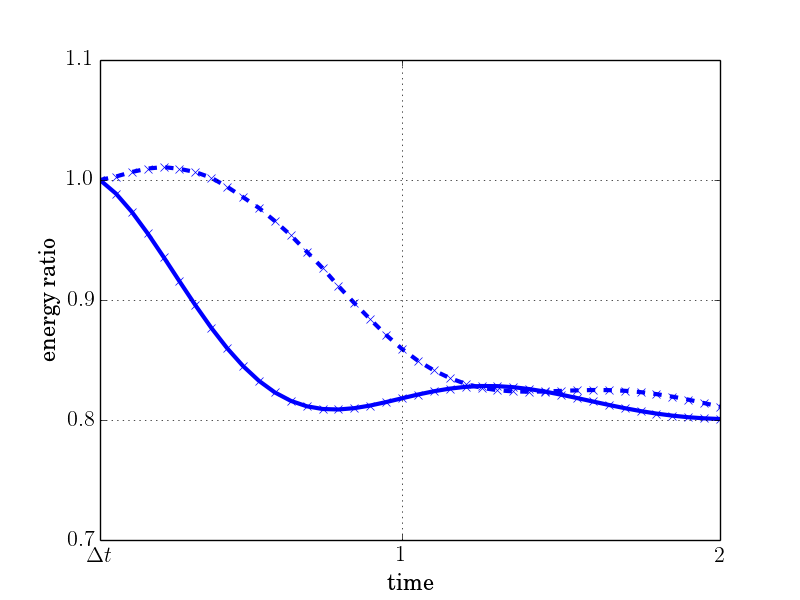}}
 \caption{Energy ratio for codimension zero structure. The structure elastic
constant $\kappa = 1$, $h_s = 1/8$,
 the fluid viscosity $\nu = 0.05$, $\delta\rho = 0.3$. The solid line
correspond to the \dlmibm scheme, while 
 the dashed line marks the energy for the \feibm scheme.}
 \label{fig:thick_energy_rho03}
 \end{figure}

\section*{Acknowledgments}

The numerical experiments presented in Section~\ref{se:num} have been
performed during our collaboration with Luca Heltai and Nicola Cavallini.
Their contribution is greatly acknowledged.

% Using BibTex is not recommended but can be handled.

\bibliographystyle{plain}
\bibliography{ref}

\begin{thebibliography}{10}

\bibitem{bbf}
D.~Boffi, F.~Brezzi, and M.~Fortin.
\newblock {\em Mixed Finite Element Methods and Applications}, volume~44 of
  {\em Springer Series in Computational Mathematics}.
\newblock Springer-Verlag, New York, 2013.

\bibitem{coupled2011}
D.~Boffi, N.~Cavallini, F.~Gardini, and L.~Gastaldi.
\newblock Immersed boundary method: Performance analysis of popular finite
  element spaces.
\newblock In M.~Papadrakakis, E.~O{\~{n}}ate, and B.~Schrefler, editors, {\em
  Computational Methods for Coupled Problems in Science and Engineering IV},
  pages 135--146, 2011.

\bibitem{bcgg2012}
D.~Boffi, N.~Cavallini, F.~Gardini, and L.~Gastaldi.
\newblock Local mass conservation of {S}tokes finite elements.
\newblock {\em J. Sci. Comput.}, 52(2):383--400, 2012.

\bibitem{bcggumi}
D.~Boffi, N.~Cavallini, F.~Gardini, and L.~Gastaldi.
\newblock Stabilized {S}tokes elements and local mass conservation.
\newblock {\em Boll. Unione Mat. Ital. (9)}, 5(3):543--573, 2012.

\bibitem{coupled2013}
D.~Boffi, N.~Cavallini, F.~Gardini, and L.~Gastaldi.
\newblock Mass preserving distributed {L}agrange multiplier approach to
  immersed boundary method.
\newblock In S.~Idelsohn, M.~Papadrakakis, and B.~Schrefler, editors, {\em
  Computational Methods for Coupled Problems in Science and Engineering V},
  pages 323--334, 2013.

\bibitem{BCG2011}
D.~Boffi, N.~Cavallini, and L.~Gastaldi.
\newblock Finite element approach to immersed boundary method with different
  fluid and solid densities.
\newblock {\em Math. Models Methods Appl. Sci.}, 21(12):2523--2550, 2011.

\bibitem{nuovodlm}
D.~Boffi, N.~Cavallini, and L.~Gastaldi.
\newblock The finite element immersed boundary method with distributed
  {L}agrange multiplier.
\newblock Submitted, 2014.

\bibitem{BGHP}
D.~Boffi, L.~Gastaldi, L.~Heltai, and C.~S. Peskin.
\newblock On the hyper-elastic formulation of the immersed boundary method.
\newblock {\em Comput. Methods Appl. Mech. Engrg.}, 197(25-28):2210--2231,
  2008.

\bibitem{bathe}
Daniele Boffi and Lucia Gastaldi.
\newblock A finite element approach for the immersed boundary method.
\newblock {\em Comput. \& Structures}, 81(8-11):491--501, 2003.
\newblock In honour of Klaus-J{\"u}rgen Bathe.

\bibitem{bgh1}
Daniele Boffi, Lucia Gastaldi, and Luca Heltai.
\newblock Numerical stability of the finite element immersed boundary method.
\newblock {\em Math. Models Methods Appl. Sci.}, 17(10):1479--1505, 2007.

\bibitem{bgh2}
Daniele Boffi, Lucia Gastaldi, and Luca Heltai.
\newblock On the {CFL} condition for the finite element immersed boundary
  method.
\newblock {\em Comput. \& Structures}, 85(11-14):775--783, 2007.

\bibitem{causin}
P.~Causin, J.~F. Gerbeau, and F.~Nobile.
\newblock Added-mass effect in the design of partitioned algorithms for
  fluid-structure problems.
\newblock {\em Comput. Methods Appl. Mech. Engrg.}, 194(42-44):4506--4527,
  2005.

\bibitem{Fauci198885}
L.J. Fauci and C.S. Peskin.
\newblock A computational model of aquatic animal locomotion.
\newblock {\em Journal of Computational Physics}, 77(1):85--108, 1988.

\bibitem{hoppe1}
Thomas Franke, Ronald H.~W. Hoppe, Christopher Linsenmann, Lothar Schmid,
  Carina Willbold, and Achim Wixforth.
\newblock Numerical simulation of the motion of red blood cells and vesicles in
  microfluidic flows.
\newblock {\em Comput. Vis. Sci.}, 14(4):167--180, 2011.

\bibitem{girglo1995}
V.~Girault and R.~Glowinski.
\newblock Error analysis of a fictitious domain method applied to a {D}irichlet
  problem.
\newblock {\em Japan J. Indust. Appl. Math.}, 12(3):487--514, 1995.

\bibitem{girglopan}
V.~Girault, R.~Glowinski, and T.~W. Pan.
\newblock A fictitious-domain method with distributed multiplier for the
  {S}tokes problem.
\newblock In {\em Applied nonlinear analysis}, pages 159--174. Kluwer/Plenum,
  New York, 1999.

\bibitem{Givelberg20040283}
E.~Givelberg.
\newblock Modeling elastic shells immersed in fluid.
\newblock {\em Communications on Pure and Applied Mathematics},
  57(3):0283--0309, 2004.

\bibitem{Givelberg2003377}
E.~Givelberg and J.~Bunn.
\newblock A comprehensive three-dimensional model of the cochlea.
\newblock {\em Journal of Computational Physics}, 191(2):377--391, 2003.

\bibitem{glokuz2007}
R.~Glowinski and Yu. Kuznetsov.
\newblock Distributed {L}agrange multipliers based on fictitious domain method
  for second order elliptic problems.
\newblock {\em Comput. Methods Appl. Mech. Engrg.}, 196(8):1498--1506, 2007.

\bibitem{Griffith2012433}
B.E. Griffith and S.~Lim.
\newblock Simulating an elastic ring with bend and twist by an adaptive
  generalized immersed boundary method.
\newblock {\em Communications in Computational Physics}, 12(2):433--461, 2012.

\bibitem{heltai}
L.~Heltai.
\newblock On the stability of the finite element immersed boundary method.
\newblock {\em Comput. \& Structures}, 86(7-8):598--617, 2008.

\bibitem{heltaicostanzo}
Luca Heltai and Francesco Costanzo.
\newblock Variational implementation of immersed finite element methods.
\newblock {\em Comput. Methods Appl. Mech. Engrg.}, 229/232:110--127, 2012.

\bibitem{Heys2008977}
J.J. Heys, T.~Gedeon, B.C. Knott, and Y.~Kim.
\newblock Modeling arthropod filiform hair motion using the penalty immersed
  boundary method.
\newblock {\em Journal of Biomechanics}, 41(5):977--984, 2008.

\bibitem{hoppe2}
Ronald H.~W. Hoppe and Christopher Linsenmann.
\newblock The finite element immersed boundary method for the numerical
  simulation of the motion of red blood cells in microfluidic flows.
\newblock In {\em Numerical methods for differential equations, optimization,
  and technological problems}, volume~27 of {\em Comput. Methods Appl. Sci.},
  pages 3--17. Springer, Dordrecht, 2013.

\bibitem{Kim2009927}
Y.~Kim, S.~Lim, S.V. Raman, O.P. Simonetti, and A.~Friedman.
\newblock Blood flow in a compliant vessel by the immersed boundary method.
\newblock {\em Annals of Biomedical Engineering}, 37(5):927--942, 2009.

\bibitem{Kim20062294}
Y.~Kim and C.S. Peskin.
\newblock 2-{D} parachute simulation by the immersed boundary method.
\newblock {\em SIAM Journal on Scientific Computing}, 28(6):2294--2312, 2006.

\bibitem{Leveque1988191}
Randall~J. Leveque, Charles~S. Peskin, and Peter~D. Lax.
\newblock Solution of a two-dimensional cochlea model with fluid viscosity.
\newblock {\em SIAM Journal on Applied Mathematics}, 48(1):191--213, 1988.

\bibitem{LiuKimTang}
W.~K. Liu, D.~W. Kim, and S.~Tang.
\newblock Mathematical foundations of the immersed finite element method.
\newblock {\em Comput. Mech.}, 39(3):211--222, 2007.

\bibitem{Miller2005195}
L.A. Miller and C.S. Peskin.
\newblock A computational fluid dynamics of 'clap and fling' in the smallest
  insects.
\newblock {\em Journal of Experimental Biology}, 208(2):195--212, 2005.

\bibitem{Pe2}
Charles~S. Peskin.
\newblock Numerical analysis of blood flow in the heart.
\newblock {\em J. Computational Phys.}, 25(3):220--252, 1977.

\bibitem{PeAN}
Charles~S. Peskin.
\newblock The immersed boundary method.
\newblock {\em Acta Numer.}, 11:479--517, 2002.

\bibitem{Pe1}
C.S. Peskin.
\newblock Flow patterns around heart valves: A numerical method.
\newblock {\em Journal of Computational Physics}, 10(2):252--271, 1972.

\bibitem{WangLiu}
X.~Wang and W.K. Liu.
\newblock Extended immersed boundary method using {FEM} and {RKPM}.
\newblock {\em Comput. Methods Appl. Mech. Engrg.}, 193:1305--1321, 2004.

\bibitem{ZGWangLiu}
L.~Zhang, A.~Gerstenberger, X.~Wang, and W.K. Liu.
\newblock Immersed finite element method.
\newblock {\em Comput. Methods Appl. Mech. Engrg.}, 193:2051--2067, 2004.

\bibitem{Zhu2002452}
L.~Zhu and C.S. Peskin.
\newblock Simulation of a flapping flexible filament in a flowing soap film by
  the immersed boundary method.
\newblock {\em Journal of Computational Physics}, 179(2):452--468, 2002.

\end{thebibliography}

\end{document}